\documentclass[11pt]{article}

\usepackage{stmaryrd}

\usepackage{amsfonts}
\usepackage{amsthm}
\usepackage{amssymb,amsmath}
\usepackage{latexsym}
\usepackage{graphics}
\usepackage{epic}
\usepackage{epsfig}
\usepackage{psfrag}
\usepackage{pict2e}  
\usepackage{subfigure}
\usepackage{bbm}
\usepackage{dsfont}
\usepackage{color}
\usepackage{tikz}
\definecolor{gray5}{gray}{0.8}
\definecolor{gray1}{gray}{0.4}
\definecolor{gray2}{gray}{0.6}
\definecolor{gray3}{gray}{0.7}
\definecolor{gray4}{gray}{0.3}

\numberwithin{equation}{section}

\newtheorem     {thm}{Theorem}[section]

\newtheorem     {lem}[thm]{Lemma}
\newtheorem     {prop}[thm]{Proposition}
\newtheorem     {cor}[thm]{Corollary}

\newtheorem     {rem}[thm]{Remark}

\begin{document}

\title{Central limit theorem for supercritical binary homogeneous Crump-Mode-Jagers processes} 
\date{}
\author{\textsc{Benoit Henry$^{1,2}$}}

\footnotetext[1]{TOSCA project-team, INRIA Nancy -- Grand Est, IECL -- UMR 7502,
  Nancy-Universit\'e, Campus scientifique, B.P.\ 70239, 54506 Vand\oe uvre-l\`es-Nancy Cedex,
  France}

\footnotetext[2]{IECL -- UMR 7502,
Nancy-Universit\'e, Campus scientifique, B.P.\ 70239, 54506 Vand\oe uvre-l\`es-Nancy Cedex,
  France, E-mail: \texttt{benoit.henry@univ-lorraine.fr}
}
\maketitle

\begin{abstract}

We consider a supercritical general branching population where the lifetimes of individuals are i.i.d.\ with arbitrary distribution and each individual gives birth to new individuals at Poisson times independently from each others. The population counting process of such population is a known as binary homogeneous Crump-Jargers-Mode process. It is known that such processes converges almost surely when correctly renormalized. In this paper, we study the error of this convergence. To this end, we use classical renewal theory and recent works \cite{L10,CL1,CH} on this model to obtain the moments of the error. Then, we can precisely study the asymptotic behaviour of these moments thanks to L\'evy processes theory. These results in conjunction with a new decomposition of the splitting trees allow us to obtain a central limit theorem.

\end{abstract}          
\bigskip



\section{Introduction}
\label{sec:intro}
In this work, we consider a general branching population where individuals live and reproduce independently from each other. Their lifetimes follow an arbitrary distribution $\mathbb{P}_{V}$ and the births occur at Poisson times with constant rate $b$. The genealogical tree induced by this population is called a \emph{splitting tree} \cite{GJ,GK,L10} and is of main importance in the study of the model.

The population counting process $N_{t}$ (giving the number of living individuals at time $t$) is a binary homogeneous Crump-Mode-Jagers (CMJ) process. Crump-Mode-Jagers processes are  very general branching processes. Such processes are known to have many applications. For instance, in biology, they have recently been used to model spreading diseases (see \cite{cmjref1,cmjref2}). Another example of application appears in queuing theory (see \cite{cmjref3} and \cite{cmjref4}).

In \cite{N}, Nerman shows very general conditions for the almost sure convergence of general CMJ processes. In the supercritical case, it is known that the quantity $e^{-\alpha t}N_{t}$, where $\alpha$ is the Malthusian parameter of the population, converges almost surely. This result has been proved in \cite{Rich} using Jagers-Nerman's theory of general branching processes counted by random characteristics. Another proof can be found in \cite{CH}, using only elementary probabilistic tools, relying on fluctuation analysis of the process. 

Our purpose in this work is to investigate the behaviour of the error in the aforementioned convergence.
Many papers studied the second order behaviour of converging branching processes. Early works investigate the Galton-Watson case. In \cite{heyde} and \cite{heyde2}, Heyde obtained rates of convergence and get central limit theorems in the case of supercritical Galton-Watson when the limit has finite variance. Later, in \cite{asm}, Asmussen obtained the polynomial convergence rates in the general case. 
In our model, the particular case when the individuals never die (i.e.\ $\mathbb{P}_{V}=\delta_{\infty}$, implying that the population counting process is a Markovian Yule process) has already been studied. More precisely, Athreya showed in \cite{ath}, for a Markovian branching process $Z$ with appropriate conditions, and such that $e^{-\alpha t}Z_{t}$ converges to some random variable $W$ a.s., that the error
\[
\frac{Z_{t}-e^{\alpha t}W}{\sqrt{Z_{t}}},
\]
converges in distribution to some Gaussian random variable. 

In the case of general CMJ processes, there was no similar result although very recent work of Iksanov and Meiners \cite{iks} gives sufficient conditions for the error terms in the convergence of supercritical general branching processes to be $o(t^{\delta})$ in a very general background (arbitrary birth point process). Although our model is more specific, we give slightly more precise results. Indeed, we give the exact rate of convergence, $e^{\frac{\alpha}{2}t}$, and characterized the limit. Moreover, we believe that our method could apply to other general branching processes counted by random characteristics, as soon as the birth point process is Poissonian.

The first step of the method is to obtain informations on the moments of the error in the a.s.\ convergence of the process. Using the renewal structure of the tree and formulae on the expectation of a random integral, we are able to express the moments of the error in terms of the scale function of a L\'evy process. This process is known to be the contour process of the splitting tree as constructed in \cite{L10}. The asymptotic behaviours of the moments are then precisely studied thanks to the analysis of the ladder height process associated to a similar L\'evy process and to the Wiener-Hopf factorization. The second ingredient is a decomposition of the splitting tree into subtrees whose laws are characterized by the overshoots of the contour process over a fixed level. Finally, the error term can be decomposed as the sum of the error made in each subtrees. Our controls on the moments ensure that the error in each subtree decreases fast enough compared to the growth of the population (see Section \ref{sec:strategy} for details).

The first section is devoted to the introduction of main tools used in this work. The first part recall the basic facts on splitting trees which are essentially borrowed from \cite{L10,CL1,CL2,CLR,CH}. The second part recall some classical facts on renewal equations and the last part gives a useful Lemma on the expectation of a random integral.
Section \ref{sec:results} is devoted to the statement of Theorem \ref{thm:tclN} which is a CLT for the population counting process $N_{t}$. Section \ref{sec:strategy} details the main lines of the method. Theorem \ref{thm:tclN} is finally proved in Section \ref{proof:th1}. 

\section{Splitting trees and preliminary results}
This section is devoted to the statement of results which are constantly used in the sequel. The first subsection presents the model and states results on splitting trees coming from \cite{L10,CL1,CL2,Rich,CH}. The second subsection recalls some well-known results on renewal equations. Finally, the last subsection is devoted to the statement and the proof of a lemma for the expectation of random integrals, which is constantly used in the sequel.
\subsection{Splitting trees}
\label{sec:models}
In this paper, we study a model of population dynamics called a splitting tree. We consider a branching tree (see Figure \ref{fig: tree}), where individuals live and reproduce independently from each other. Their lifetimes are i.i.d.\ following an arbitrary distribution $\mathbb{P}_{V}$. Given the lifetime of an individual, he gives birth to new individuals at Poisson times with  positive constant rate $b$ until his death independently from the other individuals. We also suppose that the population starts with a single individual called the {\it root}.

The finite measure $\Lambda:=b\mathbb{P}_{V}$ is called the \emph{lifespan measure}, and plays an important role in the study of the model.

\begin{figure}[ht]
\unitlength 2mm 
\linethickness{0.4pt}

\unitlength 1.5mm 
\linethickness{0.6pt}
\ifx\plotpoint\undefined\newsavebox{\plotpoint}\fi 
\begin{picture}(68,40)(-10,0)
\put(15,5){\vector(0,1){35}}
\put(13.93,5.93){\line(1,0){.9836}}
\put(15.897,5.93){\line(1,0){.9836}}
\put(17.864,5.93){\line(1,0){.9836}}
\put(19.831,5.93){\line(1,0){.9836}}
\put(21.799,5.93){\line(1,0){.9836}}
\put(23.766,5.93){\line(1,0){.9836}}
\put(25.733,5.93){\line(1,0){.9836}}
\put(27.7,5.93){\line(1,0){.9836}}
\put(29.667,5.93){\line(1,0){.9836}}
\put(31.635,5.93){\line(1,0){.9836}}
\put(33.602,5.93){\line(1,0){.9836}}
\put(35.569,5.93){\line(1,0){.9836}}
\put(37.536,5.93){\line(1,0){.9836}}
\put(39.503,5.93){\line(1,0){.9836}}
\put(41.471,5.93){\line(1,0){.9836}}
\put(43.438,5.93){\line(1,0){.9836}}
\put(45.405,5.93){\line(1,0){.9836}}
\put(47.372,5.93){\line(1,0){.9836}}
\put(49.34,5.93){\line(1,0){.9836}}
\put(51.307,5.93){\line(1,0){.9836}}
\put(53.274,5.93){\line(1,0){.9836}}
\put(55.241,5.93){\line(1,0){.9836}}
\put(57.208,5.93){\line(1,0){.9836}}
\put(59.176,5.93){\line(1,0){.9836}}
\put(61.143,5.93){\line(1,0){.9836}}
\put(63.11,5.93){\line(1,0){.9836}}
\put(65.077,5.93){\line(1,0){.9836}}
\put(67.044,5.93){\line(1,0){.9836}}
\put(69.012,5.93){\line(1,0){.9836}}
\put(70.979,5.93){\line(1,0){.9836}}
\put(72.946,5.93){\line(1,0){.9836}}
\put(19,20){\line(0,-1){14}}
\put(25,28){\line(0,-1){10}}
\put(31,23){\line(0,1){8}}
\put(36,29){\line(0,1){7}}
\put(45,26){\line(0,1){12}}
\put(53,11){\line(0,1){9}}
\put(59,16){\line(0,1){19}}
\put(24.93,17.93){\line(-1,0){.8571}}
\put(23.215,17.93){\line(-1,0){.8571}}
\put(21.501,17.93){\line(-1,0){.8571}}
\put(19.787,17.93){\line(-1,0){.8571}}
\put(30.93,22.93){\line(-1,0){.8571}}
\put(29.215,22.93){\line(-1,0){.8571}}
\put(27.501,22.93){\line(-1,0){.8571}}
\put(25.787,22.93){\line(-1,0){.8571}}
\put(35.93,28.93){\line(-1,0){.8333}}
\put(34.263,28.93){\line(-1,0){.8333}}
\put(32.596,28.93){\line(-1,0){.8333}}
\put(44.93,25.93){\line(-1,0){.9333}}
\put(43.063,25.93){\line(-1,0){.9333}}
\put(41.196,25.93){\line(-1,0){.9333}}
\put(39.33,25.93){\line(-1,0){.9333}}
\put(37.463,25.93){\line(-1,0){.9333}}
\put(35.596,25.93){\line(-1,0){.9333}}
\put(33.73,25.93){\line(-1,0){.9333}}
\put(31.863,25.93){\line(-1,0){.9333}}
\put(58.93,15.93){\line(-1,0){.8571}}
\put(57.215,15.93){\line(-1,0){.8571}}
\put(55.501,15.93){\line(-1,0){.8571}}
\put(53.787,15.93){\line(-1,0){.8571}}
\put(52.93,10.93){\line(-1,0){.9962}}
\put(50.937,10.93){\line(-1,0){.9962}}
\put(48.945,10.93){\line(-1,0){.9962}}
\put(46.953,10.93){\line(-1,0){.9962}}
\put(44.96,10.93){\line(-1,0){.9962}}
\put(42.968,10.93){\line(-1,0){.9962}}
\put(40.976,10.93){\line(-1,0){.9962}}
\put(38.983,10.93){\line(-1,0){.9962}}
\put(36.991,10.93){\line(-1,0){.9962}}
\put(34.999,10.93){\line(-1,0){.9962}}
\put(33.006,10.93){\line(-1,0){.9962}}
\put(31.014,10.93){\line(-1,0){.9962}}
\put(29.021,10.93){\line(-1,0){.9962}}
\put(27.029,10.93){\line(-1,0){.9962}}
\put(25.037,10.93){\line(-1,0){.9962}}
\put(23.044,10.93){\line(-1,0){.9962}}
\put(21.052,10.93){\line(-1,0){.9962}}
\put(67,23){\line(0,1){5}}
\put(73,25){\line(0,1){11}}
\put(66,12){\line(0,1){4}}
\put(65.93,11.93){\line(-1,0){.9286}}
\put(64.073,11.93){\line(-1,0){.9286}}
\put(62.215,11.93){\line(-1,0){.9286}}
\put(60.358,11.93){\line(-1,0){.9286}}
\put(58.501,11.93){\line(-1,0){.9286}}
\put(56.644,11.93){\line(-1,0){.9286}}
\put(54.787,11.93){\line(-1,0){.9286}}
\put(66.93,22.93){\line(-1,0){.8889}}
\put(65.152,22.93){\line(-1,0){.8889}}
\put(63.374,22.93){\line(-1,0){.8889}}
\put(61.596,22.93){\line(-1,0){.8889}}
\put(59.819,22.93){\line(-1,0){.8889}}
\put(72.93,24.93){\line(-1,0){.8571}}
\put(71.215,24.93){\line(-1,0){.8571}}
\put(69.501,24.93){\line(-1,0){.8571}}
\put(67.787,24.93){\line(-1,0){.8571}}
\put(13,39){\makebox(0,0)[cc]{$t$}}

\end{picture}
\caption{ Graphical representation of a Splitting tree. The vertical axis represents the biological time and the horizontal axis has no biological meaning. The vertical lines represent the individuals, their lengths correspond to their lifetimes. The dashed lines denote the filiations between individuals. }
\label{fig: tree}
\end{figure}
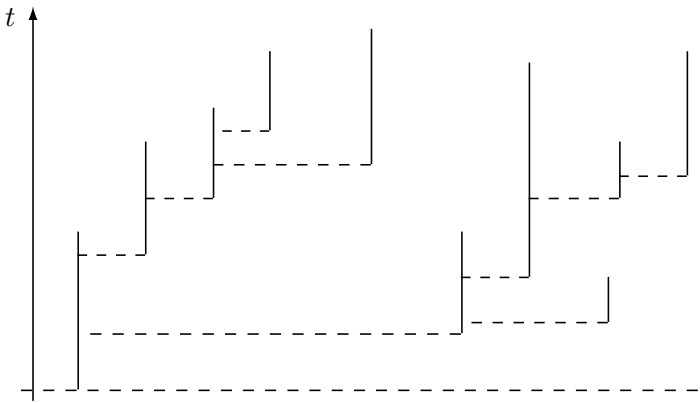

In \cite{L10}, Lambert introduces a contour process $Y$, which codes for the splitting tree. Suppose we are given a tree $\mathbb{T}$, seen as a subset of $\mathbb{R}\times\left(\cup_{k\geq 0}\mathbb{N}^{k}\right)$ with some compatibility conditions (see \cite{L10}). On this object, Lambert constructs a Lebesgue measure $\lambda$ and a total order relation  $\preceq $ which can be roughly summarized as follows: let $x,y$ in $\mathbb{T}$, the point of birth of the lineage of $x$ during the lifetime of the root split the tree in two connected components, then $y\preceq x$ if $y$ belong to the same component as $x$ but is not an ancestor of $x$ (see Figure \ref{fig:order}).

If we assume that $\lambda(\mathbb{T})$ is finite, then the application,
\[
\begin{array}{cccc}
\varphi:&\mathbb{T} & \rightarrow & [0,\lambda\left(\mathbb{T}\right)], \\
& x & \rightarrow & \lambda\left(\{y  \mid  \ y\preceq x\} \right), \\
\end{array}
\]
is a bijection. Moreover, in a graphical sens (see Figure \ref{fig:order}), $\varphi(x)$ measures the length of the part of the tree which is above the lineage of $x$.
\begin{figure}[ht]
\unitlength 2mm 
\linethickness{0.4pt}

\unitlength 1.7mm 
\linethickness{0.8pt}
\begin{picture}(15,45)(-30,0)

\put(0,0){\line(0,1){15}}
\color{gray5}
\put(0,15){\line(0,1){15}}
\color{black}
\put(0,0){\line(1,0){2}}
\put(0,30){\line(1,0){2}}
\multiput(0,15)(1,0){3}{\line(1,0){0.5}}

\put(3.5,15){\line(0,1){5}}
\color{gray5}
\put(3.5,20){\line(0,1){15}}
\color{black}
\put(2.5,15){\line(1,0){2}}
\put(2.5,35){\line(1,0){2}}
\multiput(3.5,30)(1,0){3}{\line(1,0){0.5}}
\color{gray5}
\put(6.5,30){\line(0,1){10}}
\color{black}
\put(5.5,30){\line(1,0){2}}
\put(5.5,40){\line(1,0){2}}
\multiput(6.5,35)(1,0){3}{\line(1,0){0.5}}
\color{gray5}
\put(9.5,35){\line(0,1){10}}
\color{black}
\put(8.5,35){\line(1,0){2}}
\put(8.5,45){\line(1,0){2}}
\multiput(6.5,33)(1,0){6}{\line(1,0){0.5}}
\color{gray5}
\put(12.5,33){\line(0,1){5}}
\color{black}
\put(11.5,33){\line(1,0){2}}
\put(11.5,38){\line(1,0){2}}
\multiput(12.5,34)(1,0){3}{\line(1,0){0.5}}
\color{gray5}
\put(15.5,34){\line(0,1){10}}
\color{black}
\put(14.5,34){\line(1,0){2}}
\put(14.5,44){\line(1,0){2}}
\put(19.5,22){\circle*{1}}
\put(20.5,22){\makebox(1,1){$x$}}
\multiput(3.5,20)(1,0){15}{\line(1,0){0.5}}
\color{gray5}
\put(19.5,22){\line(0,1){8}}
\color{black}
\put(19.5,20){\line(0,1){2}}

\put(18.5,20){\line(1,0){2}}
\put(18.5,30){\line(1,0){2}}
\multiput(19.5,28)(1,0){3}{\line(1,0){0.5}}
\color{gray5}
\put(23.5,28){\line(0,1){15}}
\color{black}
\put(22.5,28){\line(1,0){2}}
\put(22.5,43){\line(1,0){2}}
\multiput(23.5,38)(1,0){3}{\line(1,0){0.5}}
\color{gray5}
\put(27.5,38){\line(0,1){6}}
\color{black}
\put(26.5,38){\line(1,0){2}}
\put(26.5,44){\line(1,0){2}}
\multiput(0,5)(1,0){31}{\line(1,0){0.5}}
\put(31.5,5){\line(0,1){15}}
\put(30.5,5){\line(1,0){2}}
\put(30.5,20){\line(1,0){2}}
\multiput(31.5,17)(1,0){3}{\line(1,0){0.5}}
\put(35.5,17){\line(0,1){20}}
\put(34.5,17){\line(1,0){2}}
\put(34.5,37){\line(1,0){2}}
\multiput(35.5,33)(1,0){3}{\line(1,0){0.5}}
\put(39.5,33){\line(0,1){10}}
\put(38.5,33){\line(1,0){2}}
\put(38.5,43){\line(1,0){2}}
\multiput(31.5,7)(1,0){10}{\line(1,0){0.5}}
\put(42.5,7){\line(0,1){10}}
\put(41.5,7){\line(1,0){2}}
\put(41.5,17){\line(1,0){2}}
\multiput(42.5,10)(1,0){3}{\line(1,0){0.5}}
\put(46.5,10){\line(0,1){4}}
\put(45.5,10){\line(1,0){2}}
\put(45.5,14){\line(1,0){2}}
\end{picture}
\caption{In gray, the set $\left\{y\in\mathbb{T}\mid y\preceq x \right\}$ }
\label{fig:order}
\end{figure}
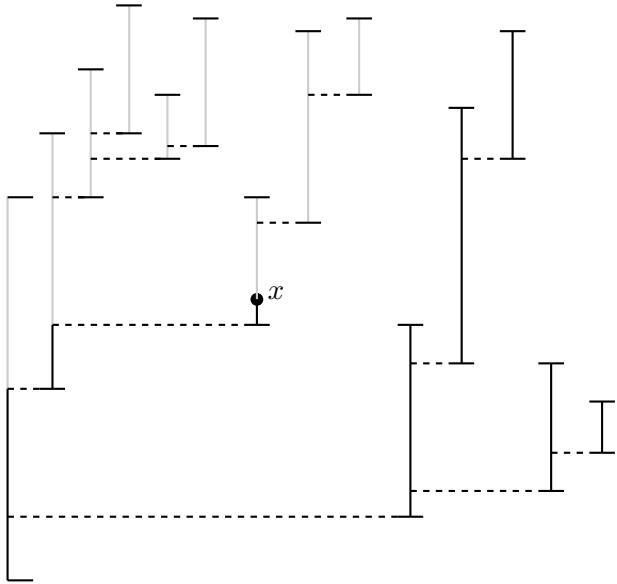
The contour process is then defined, for all $s$, by, 
\[
Y_{s}:=\Pi_{\mathbb{R}}\left(\varphi^{-1}\left(s\right)\right),
\]
where $\Pi_{\mathbb{R}}$ is the projection from $\mathbb{R}\times\left(\cup_{k\geq 0}\mathbb{N}^{k}\right)$ to $\mathbb{R}$. 

In a more graphical way, the contour process can be seen as the graph of an exploration process of the tree:
it begins at the top of the root and decreases with slope $-1$ while running back along the life of the root until it meets a birth. The contour process then jumps at the top of the life interval of the child born at this time and continues its exploration as before. If the exploration process does not encounter a birth when exploring the life interval of an individual, it goes back to its parent and continues the exploration from the birth-date of the just left individual (see Figure \ref{fig:contour}).
It is then readily seen that the intersections of the contour process with the line of ordinate $t$ are in one-to-one correspondence with the individuals in the tree alive at time $t$. 

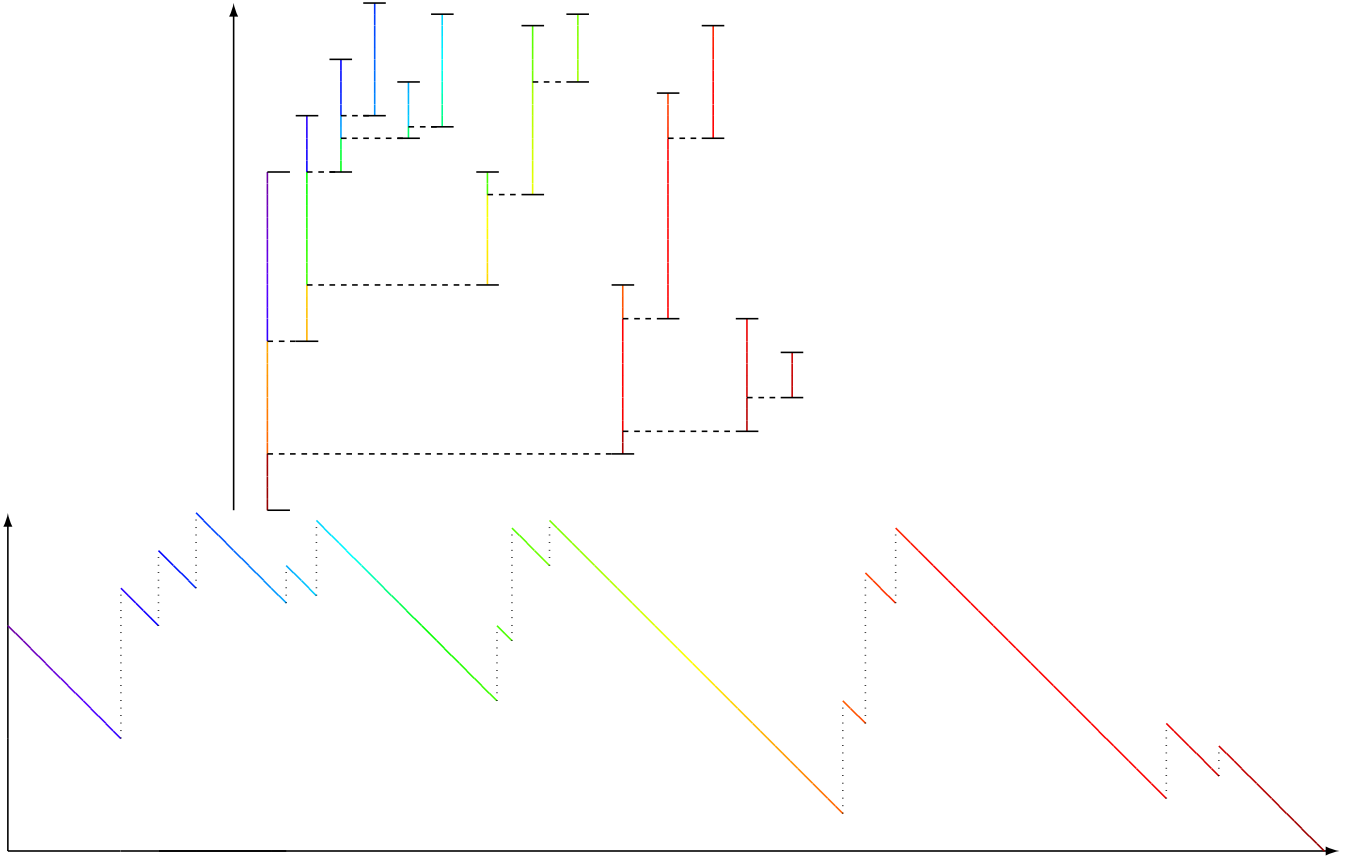
\begin{figure}[ht]
\unitlength 1.5mm 
\linethickness{0.6pt}
\begin{picture}(250,50)(-20,0)
\put(0,0){\vector(0,1){45}}
\put(3,0){
\newcount\compteur \compteur=400
\multiput(0,30)(0,-1){15}{\color[wave]{\the\compteur}\global\advance\compteur by 2\line(0,-1){1}}
\multiput(3.5,35)(0,-1){5}{\color[wave]{\the\compteur}\global\advance\compteur by 2\line(0,-1){1}}
\multiput(6.5,40)(0,-1){5}{\color[wave]{\the\compteur}\global\advance\compteur by 2\line(0,-1){1}}
\multiput(9.5,45)(0,-1){10}{\color[wave]{\the\compteur}\global\advance\compteur by 2\line(0,-1){1}}
\multiput(6.5,35)(0,-1){2}{\color[wave]{\the\compteur}\global\advance\compteur by 2\line(0,-1){1}}
\multiput(12.5,38)(0,-1){4}{\color[wave]{\the\compteur}\global\advance\compteur by 2\line(0,-1){1}}
\multiput(15.5,44)(0,-1){10}{\color[wave]{\the\compteur}\global\advance\compteur by 2\line(0,-1){1}}
\multiput(12.5,34)(0,-1){1}{\color[wave]{\the\compteur}\global\advance\compteur by 2\line(0,-1){1}}
\multiput(6.5,33)(0,-1){3}{\color[wave]{\the\compteur}\global\advance\compteur by 2\line(0,-1){1}}
\multiput(3.5,30)(0,-1){10}{\color[wave]{\the\compteur}\global\advance\compteur by 2\line(0,-1){1}}
\multiput(19.5,30)(0,-1){2}{\color[wave]{\the\compteur}\global\advance\compteur by 2\line(0,-1){1}}
\multiput(23.5,43)(0,-1){5}{\color[wave]{\the\compteur}\global\advance\compteur by 2\line(0,-1){1}}
\multiput(27.5,44)(0,-1){6}{\color[wave]{\the\compteur}\global\advance\compteur by 2\line(0,-1){1}}
\multiput(23.5,38)(0,-1){10}{\color[wave]{\the\compteur}\global\advance\compteur by 2\line(0,-1){1}}
\multiput(19.5,28)(0,-1){8}{\color[wave]{\the\compteur}\global\advance\compteur by 2\line(0,-1){1}}
\multiput(3.5,20)(0,-1){5}{\color[wave]{\the\compteur}\global\advance\compteur by 2\line(0,-1){1}}
\multiput(0,15)(0,-1){10}{\color[wave]{\the\compteur}\global\advance\compteur by 2\line(0,-1){1}}
\multiput(31.5,20)(0,-1){3}{\color[wave]{\the\compteur}\global\advance\compteur by 2\line(0,-1){1}}
\multiput(35.5,37)(0,-1){4}{\color[wave]{\the\compteur}\global\advance\compteur by 2\line(0,-1){1}}
\multiput(39.5,43)(0,-1){10}{\color[wave]{\the\compteur}\global\advance\compteur by 2\line(0,-1){1}}
\multiput(35.5,33)(0,-1){16}{\color[wave]{\the\compteur}\global\advance\compteur by 2\line(0,-1){1}}
\multiput(31.5,17)(0,-1){10}{\color[wave]{\the\compteur}\global\advance\compteur by 2\line(0,-1){1}}
\multiput(42.5,17)(0,-1){7}{\color[wave]{\the\compteur}\global\advance\compteur by 2\line(0,-1){1}}
\multiput(46.5,14)(0,-1){4}{\color[wave]{\the\compteur}\global\advance\compteur by 2\line(0,-1){1}}
\multiput(42.5,10)(0,-1){3}{\color[wave]{\the\compteur}\global\advance\compteur by 2\line(0,-1){1}}
\multiput(31.5,7)(0,-1){2}{\color[wave]{\the\compteur}\global\advance\compteur by 2\line(0,-1){1}}
\multiput(0,5)(0,-1){5}{\color[wave]{\the\compteur}\global\advance\compteur by 2\line(0,-1){1}}
\put(0,0){\line(1,0){2}}
\put(0,30){\line(1,0){2}}
\multiput(0,15)(1,0){3}{\line(1,0){0.5}}
\put(2.5,15){\line(1,0){2}}
\put(2.5,35){\line(1,0){2}}
\multiput(3.5,30)(1,0){3}{\line(1,0){0.5}}
\put(5.5,30){\line(1,0){2}}
\put(5.5,40){\line(1,0){2}}
\multiput(6.5,35)(1,0){3}{\line(1,0){0.5}}
\put(8.5,35){\line(1,0){2}}
\put(8.5,45){\line(1,0){2}}
\multiput(6.5,33)(1,0){6}{\line(1,0){0.5}}
\put(11.5,33){\line(1,0){2}}
\put(11.5,38){\line(1,0){2}}
\multiput(12.5,34)(1,0){3}{\line(1,0){0.5}}
\put(14.5,34){\line(1,0){2}}
\put(14.5,44){\line(1,0){2}}
\multiput(3.5,20)(1,0){15}{\line(1,0){0.5}}
\put(18.5,20){\line(1,0){2}}
\put(18.5,30){\line(1,0){2}}
\multiput(19.5,28)(1,0){3}{\line(1,0){0.5}}
\put(22.5,28){\line(1,0){2}}
\put(22.5,43){\line(1,0){2}}
\multiput(23.5,38)(1,0){3}{\line(1,0){0.5}}
\put(26.5,38){\line(1,0){2}}
\put(26.5,44){\line(1,0){2}}
\multiput(0,5)(1,0){31}{\line(1,0){0.5}}
\put(30.5,5){\line(1,0){2}}
\put(30.5,20){\line(1,0){2}}
\multiput(31.5,17)(1,0){3}{\line(1,0){0.5}}
\put(34.5,17){\line(1,0){2}}
\put(34.5,37){\line(1,0){2}}
\multiput(35.5,33)(1,0){3}{\line(1,0){0.5}}
\put(38.5,33){\line(1,0){2}}
\put(38.5,43){\line(1,0){2}}
\multiput(31.5,7)(1,0){10}{\line(1,0){0.5}}
\put(41.5,7){\line(1,0){2}}
\put(41.5,17){\line(1,0){2}}
\multiput(42.5,10)(1,0){3}{\line(1,0){0.5}}
\put(45.5,10){\line(1,0){2}}
\put(45.5,14){\line(1,0){2}}
}
\end{picture}
\unitlength 1mm
\begin{picture}(160,45)
\put(20,0){\vector(1,0){157}}
\put(0,15){\line(0,1){15}}

\put(0,0){\line(1,0){15}}

\put(15,0){\line(1,0){5}}

\put(20,0){\line(1,0){5}}

\put(25,0){\line(1,0){10}}

\put(35,0){\line(1,0){2}}
\put(0,0){\line(0,1){15}}
\put(0,30){\vector(0,1){15}}
\newcount\compteur \compteur=400
\multiput(0,30)(1,-1){15}{\color[wave]{\the\compteur}\global\advance\compteur by 2\line(1,-1){1}} 
\multiput(15,15)(0,1){20}{\line(0,1){0.1}}
\multiput(15,35)(1,-1){5}{\color[wave]{\the\compteur}\global\advance\compteur by 2\line(1,-1){1}}
\multiput(20,30)(0,1){10}{\line(0,1){0.1}}
\multiput(20,40)(1,-1){5}{\color[wave]{\the\compteur}\global\advance\compteur by 2\line(1,-1){1}} 
\multiput(25,35)(0,1){10}{\line(0,1){0.1}}
\multiput(25,45)(1,-1){10}{\color[wave]{\the\compteur}\global\advance\compteur by 2\line(1,-1){1}} 
\multiput(35,35)(1,-1){2}{\color[wave]{\the\compteur}\global\advance\compteur by 2\line(1,-1){1}}
\multiput(37,33)(0,1){5}{\line(0,1){0.1}}
\multiput(37,38)(1,-1){4}{\color[wave]{\the\compteur}\global\advance\compteur by 2\line(1,-1){1}}
\multiput(41,34)(0,1){10}{\line(0,1){0.1}}
\multiput(41,44)(1,-1){10}{\color[wave]{\the\compteur}\global\advance\compteur by 2\line(1,-1){1}}
\multiput(51,34)(1,-1){1}{\color[wave]{\the\compteur}\global\advance\compteur by 2\line(1,-1){1}}
\multiput(52,33)(1,-1){3}{\color[wave]{\the\compteur}\global\advance\compteur by 2\line(1,-1){1}}
\multiput(55,30)(1,-1){10}{\color[wave]{\the\compteur}\global\advance\compteur by 2\line(1,-1){1}}
\multiput(65,20)(0,1){10}{\line(0,1){0.1}}
\multiput(65,30)(1,-1){2}{\color[wave]{\the\compteur}\global\advance\compteur by 2\line(1,-1){1}}
\multiput(67,28)(0,1){15}{\line(0,1){0.1}}
\multiput(67,43)(1,-1){5}{\color[wave]{\the\compteur}\global\advance\compteur by 2\line(1,-1){1}}
\multiput(72,38)(0,1){6}{\line(0,1){0.1}}
\multiput(72,44)(1,-1){6}{\color[wave]{\the\compteur}\global\advance\compteur by 2\line(1,-1){1}}
\multiput(78,38)(1,-1){10}{\color[wave]{\the\compteur}\global\advance\compteur by 2\line(1,-1){1}}
\multiput(88,28)(1,-1){8}{\color[wave]{\the\compteur}\global\advance\compteur by 2\line(1,-1){1}}
\multiput(96,20)(1,-1){5}{\color[wave]{\the\compteur}\global\advance\compteur by 2\line(1,-1){1}}
\multiput(101,15)(1,-1){10}{\color[wave]{\the\compteur}\global\advance\compteur by 2\line(1,-1){1}}
\multiput(111,5)(0,1){15}{\line(0,1){0.1}}
\multiput(111,20)(1,-1){3}{\color[wave]{\the\compteur}\global\advance\compteur by 2\line(1,-1){1}}
\multiput(114,17)(0,1){20}{\line(0,1){0.1}}
\multiput(114,37)(1,-1){4}{\color[wave]{\the\compteur}\global\advance\compteur by 2\line(1,-1){1}}
\multiput(118,33)(0,1){10}{\line(0,1){0.1}}
\multiput(118,43)(1,-1){10}{\color[wave]{\the\compteur}\global\advance\compteur by 2\line(1,-1){1}}
\multiput(128,33)(1,-1){16}{\color[wave]{\the\compteur}\global\advance\compteur by 2\line(1,-1){1}}
\multiput(144,17)(1,-1){10}{\color[wave]{\the\compteur}\global\advance\compteur by 2\line(1,-1){1}}
\multiput(154,7)(0,1){10}{\line(0,1){0.1}}
\multiput(154,17)(1,-1){7}{\color[wave]{\the\compteur}\global\advance\compteur by 2\line(1,-1){1}}
\multiput(161,10)(0,1){4}{\line(0,1){0.1}}
\multiput(161,14)(1,-1){4}{\color[wave]{\the\compteur}\global\advance\compteur by 2\line(1,-1){1}}

\multiput(165,10)(1,-1){3}{\color[wave]{\the\compteur}\global\advance\compteur by 2\line(1,-1){1}}
\multiput(168,7)(1,-1){2}{\color[wave]{\the\compteur}\global\advance\compteur by 2\line(1,-1){1}}
\multiput(170,5)(1,-1){5}{\color[wave]{\the\compteur}\global\advance\compteur by 2\line(1,-1){1}}

\end{picture}
\caption{One-to-one correspondence between the tree and the graph of the contour represented by corresponding colours.}
\label{fig:contour}
\end{figure}

In the case where $\lambda(\mathbb{T})$ is infinite, one has to consider the truncations of the tree above fixed levels in order to have well-defined contours (see \cite{L10} for more details). In \cite{L10}, Lambert shows that the contour process $\left(Y^{(t)}_{s},\ s\in\mathbb{R}_{+} \right)$ of a splitting tree which has been pruned from every part above $t$ (called truncated tree above $t$), has the law of a spectrally positive L\'evy process started at the lifespan $V$ of the root, reflected below $t$ and killed at $0$, with Laplace exponent $\psi$ given by
\begin{equation}
\label{eq:laplace}
\psi(x)=x-\int_{(0,\infty]}\left(1-e^{-rx}\right)\Lambda(dr), \ \ x\in\mathbb{R}_{+}.
\end{equation}
In particular, the Laplace transform of $\mathbb{P}_{V}$ can be expressed in terms of $\psi$,
\begin{equation}
\label{eq:laplacepv}
\int_{\mathbb{R}_{+}}e^{-\lambda v}\mathbb{P}_{V}(dv)=1+\frac{\psi(\lambda)-\lambda}{b}.
\end{equation}
The largest root of $\psi$, denoted $\alpha$, characterizes the way the population expend. In this paper, we only investigate the behavior of the population in the supercritical case, when $\alpha>0$. In particular, using the convexity of $\psi$ (see \cite{Kyp}), this is equivalent to $\psi'(0+)<0$. Now, since
\begin{equation}
\label{eq:laplaceDerivative}
\psi'(x)=1-\int_{\mathbb{R}_{+}}xe^{-xv}\ b\mathbb{P}_{V}(dv),\quad \forall x\in\mathbb{R}_{+},
\end{equation}
one can see that the condition $\psi'(0+)<0$ is also equivalent to have $b\mathbb{E}\left[V \right]>1$ which is a more usual supercritical condition.
In the supercritical case, the population grows exponentially fast on the survival event with rate $\alpha$. 
According to \eqref{eq:laplacepv}, one can also see that
\begin{equation}
\label{eq:intalpha}
\int_{\mathbb{R}_{+}}e^{-\alpha v}\mathbb{P}_{V}(dv)=1-\frac{\alpha}{b}.
\end{equation}

As said earlier, an important feature of the contour process is that the number of alive individuals at time $t$ equals \[
\text{Card}\{Y_{s}^{(t)}=t \mid s\in\mathbb{R}_{+} \}.
\]
This set is the number of times the contour process hits $t$. This allows getting, thanks to the theory of L\'evy processes, the law of the unidimensional marginals of the process $\left(N_{t}, \ t\in\mathbb{R}_{+} \right)$. Indeed, let $\tau_{t}$ (resp. $\tau_{0}$) be the hitting time of $t$ (resp. of $0$) by the contour process $Y^{(t)}$. Now, for any positive integer $k$, the strong Markov property entails that
\[
\mathbb{P}\left(N_{t}=k\mid N_{t}>0 \right)=\mathbb{E}\left\{\mathbb{P}_{t\wedge V}\left(\sharp\{Y_{s}^{(t)}=t \mid s\geq0 \}=k \mid \tau_{t}<\tau_{0} \right)\right\}=\mathbb{P}_{t}\left( \sharp\{Y_{s}^{(t)}=t \mid s>0 \}=k-1 \right).
\]
Once again, the strong Markov property gives
\begin{eqnarray*}
\mathbb{P}_{t}\left( \sharp\{Y_{s}^{(t)}=t \mid s>0 \}=k-1 \right)&=&\mathbb{P}_{t}\left(\tau_{t}<\tau_{0} \right)\mathbb{P}_{t}\left(\sharp\{Y_{s}^{(t)}=t \mid s>0 \}=k-2 \right)\\&=&\mathbb{P}_{t}\left(\tau_{t}<\tau_{0} \right)^{k-1}\mathbb{P}_{t}\left(\sharp\{Y_{s}^{(t)}=t \mid s>0 \}=0 \right).
\end{eqnarray*}
Now, using fluctuation identities for spectrally positive L\'evy processes (see Theorem 8.1 in \cite{Kyp} for the spectrally negative case), we have that
\[
\mathbb{P}_{t}\left(\tau_{t}<\tau_{0} \right)=1-\frac{1}{W(t)},
\]
where $W$ is the scale function of the L\'evy process whose Laplace exponent is given by \eqref{eq:laplace}. The function $W$ is the unique increasing function whose Laplace transform is given by
\begin{equation}
\label{eq:laplacescale}
T_{\mathcal{L}}W(t)=\int_{(0,\infty)}e^{-rt}W(r)dr=\frac{1}{\psi(t)}, \quad t>\alpha,
\end{equation}
where $\alpha$ is the largest root of $\psi$.

 From the discussion above, we see that $N_{t}$ is a geometric random variable conditionally on $\left\{N_{t}>0\right\}$. More precisely, for a positive integer $k$,
 \begin{equation}
 \label{eq:loiNt}
 \mathbb{P}\left(N_{t}=k\mid N_{t}>0\right)=\frac{1}{W(t)}\left(1-\frac{1}{W(t)}\right)^{k-1}.
 \end{equation}
 In particular,
  \begin{equation}
  \label{eq:momNt}
  \mathbb{E}\left[N_{t}\mid N_{t}>0\right]=W(t).
  \end{equation}
 Moreover, it can be showed (see \cite{Rich}), that
 \begin{equation}
 \label{eq:NtNoCond}
 \mathbb{E}N_{t}=W(t)-W\star\mathbb{P}_{V}(t),
 \end{equation}
 and
 \begin{equation}
 \label{eq:probaNoCond}
 \mathbb{P}\left(N_{t}>0 \right)=1-\frac{W\star\mathbb{P}_{V}(t)}{W(t)},
 \end{equation}
 where
 \[
 W\star\mathbb{P}_{V}(t):=\int_{[0,t]}W(t-s)\mathbb{P}_{V}(ds).
 \]

 For the rest of this paper, unless otherwise stated, the notation $\mathbb{P}_{t}$ refers to $\mathbb{P}\left(.\mid N_{t}>0 \right)$ whereas $\mathbb{P}_{\infty}$ refers to the probability measure conditioned on the non-extinction event (which has positive probability in the supercritical case).

 Finally, we recall the asymptotic behaviour of the scale function $W(t)$ which is widely used in the sequel,
 \begin{lem}(\cite[Thm. 3.21]{CL1})
 \label{lem: asyComp}
 There exist a positive constant $\gamma$ such that,
 \[
 e^{-\alpha t}\psi'(\alpha)W(t)-1=\mathcal{O}\left(e^{-\gamma t} \right).
 \]

 \end{lem}
 From this Lemma and \eqref{eq:probaNoCond}, one can easily deduce that
 \begin{equation}
 \label{eq:Prnonex} 
 \mathbb{P}\left(\text{NonEx} \right)=\lim\limits_{t\to\infty}\mathbb{P}\left(N_{t}>0 \right)=\frac{\alpha}{b},
 \end{equation}
 where $\text{NonEx}$ refer to the non-extinction event.

To end this section, let us recall the law of large number for $N_{t}$.

 \begin{thm}
 \label{thm:ASCVNt}
 There exists a random variable $\mathcal{E}$, such that
 \[
e^{-\alpha t}N_{t}\underset{t\to\infty}{\rightarrow}\frac{\mathcal{E}}{\psi'(\alpha)},\quad a.s. \text{ and in } L^{2}.
 \]
 Moreover, under $\mathbb{P}_{\infty}$, $\mathcal{E}$ is exponentially distributed with parameter one.
 \end{thm}

 \subsection{A bit of renewal theory}
 \label{ssec:renew}
  The purpose of this part is to recall some facts on renewal equations borrowed from \cite{Fel}.
  Let $h:\mathbb{R}\to\mathbb{R}$ be a function bounded on finite intervals with support in $\mathbb{R}_{+}$ and $\Gamma$ a probability measure on $\mathbb{R}_{+}$.
  The equation
  \[
  F(t)=\int_{\mathbb{R}_{+}}F(t-s)\Gamma(ds)+h(t),
  \]
  called a renewal equation, is known to admit a unique solution finite on bounded interval.
 
 Here, our interest is focused on the asymptotic behavior of $F$. We said that the function $h$ is DRI (directly Riemann integrable) if for any $\delta>0$, the quantities
 \[
 \delta\sum_{i=0}^{n}\sup_{t\in[\delta i,\delta(i+1))}f(t)
 \]
 and
 \[
 \delta\sum_{i=0}^{n}\inf_{t\in[\delta i,\delta(i+1))}f(t)
 \]
 converge as $n$ goes to infinity respectively to some real numbers $I_{sup}^{\delta}$ and $I_{inf}^{\delta}$,
 and
 \[
 \lim\limits_{\delta\to 0}I_{sup}^{\delta}=\lim\limits_{\delta\to 0}I_{inf}^{\delta}<\infty.
 \]
In the sequel, we use the two following criteria for the DRI property:
 \begin{lem}
 \label{lem:DRI}
 Let $h$ a function as defined previously. If $h$ satisfies one of the next two conditions, then $h$ is DRI:
 \begin{itemize}
 \item[1.] $h$ is non-negative decreasing and classically Riemann integrable on $\mathbb{R}_{+}$,
 \item[2.]  $h$ is c\`adl\`ag and bounded by a DRI function.
 \end{itemize}
 \end{lem}
 We can now state the next result, which is constantly used in the sequel.
 \begin{thm}
 \label{thm:nonLatticeRen}
  Suppose that $\Gamma$ is non-lattice, and $h$ is DRI, then
 \[
 \lim\limits_{t\to\infty}F(t)=\gamma\int_{\mathbb{R}_{+}}h(s)ds,
 \]
 with \[
 \gamma:=\left(\int_{\mathbb{R}_{+}}s\ \Gamma(ds)\right)^{-1},
 \]
 if the above integral is finite, and zero otherwise.

 \end{thm}
 \begin{rem}
 \label{rem:1}
In particular, if we suppose that $\Gamma$ is a measure with mass lower than $1$, and that there exists a constant $\alpha\geq0$ such that
  \[
  \int_{\mathbb{R}_{+}}e^{\alpha t}\Gamma(dt)=1,
  \]
  then, one can perform the change a measure
  \[
  \widetilde{\Gamma}(dt)=e^{\alpha t}\Gamma(dt),
  \]
  in order to apply Theorem  \ref{thm:nonLatticeRen} to a new renewal equation to obtain the asymptotic behavior of $F$. (See \cite{Fel} for details).
  This method is also used in the sequel.
  \end{rem}
 \subsection{A lemma on the expectation of a random integral with respect to a Poisson random measure}
 \begin{lem}
 \label{lem:stocInt}
 Let $\xi$ be a Poisson random measure on $\mathbb{R}_{+}$ with intensity $\theta\lambda(da)$ where $\theta$ is a positive real number and $\lambda$ the Lebesgue measure. Let also $\left(X^{(i)}_{u},\ u\in\mathbb{R}_{+}\right)_{i\geq 1}$ be an i.i.d. sequence of non-negative c\`adl\`ag random processes independent of $\xi$. Let also $Y$ be a random variable independent of $\xi$ and from the family $\left(X^{(i)}_{u},\ u\in\mathbb{R}_{+}\right)_{i\geq 1}$. If $\xi_{u}$ denotes $\xi\left([0,u] \right)$, then, for any $t\geq0$,
 \[
 \mathbb{E}\int_{[0,t]}X^{(\xi_{u})}_{u}\mathds{1}_{Y>u}\ \xi(du)=\int_{0}^{t}\mathbb{P}\left(Y>u \right)\theta\mathbb{E}X_{u}  du,
 \]
 where $\left(X_{u},\ u\in\mathbb{R}_{+}\right)=\left(X^{(1)}_{u},\ u\in\mathbb{R}_{+}\right)$. In addition, for any $t\leq s$, we have
 \begin{multline*}
 \mathbb{E}\left[\int_{[0,t]} X_{v}^{(\xi_{v})}\mathds{1}_{Y>v}\ \xi(dv)\int_{[0,s]} X_{u}^{(\xi_{u})} \mathds{1}_{Y>u}\ \xi(du)\right]=\int_{0}^{t}\theta \mathbb{E}\left[X_{u}^{2} \right] \mathbb{P}\left(Y>u \right)\ du\\
 +\int_{0}^{t}\int_{0}^{s}\theta^{2}\mathbb{E}X_{u}\mathbb{E}X_{v}\mathbb{P}\left(Y>u,Y>v \right)\ dudv.
 \end{multline*}
 \end{lem}
 \begin{proof}
 Since the proof the two formulas lies on the same ideas, we only give the proof of the second equation.
 
 First of all, let $f:\mathbb{R}_{+}^{2}\to\mathbb{R}_{+}$ be a positive measurable deterministic function. We recall that, for a Poisson random measure, the measures of two disjoint measurable sets are independent random variables. That is, for $A,B$ in the Borel $\sigma$-field of $\mathbb{R}_{+}$, $\xi(A\cap B^{c})$ is independent from $\xi(B)$, which leads to
 \[
 \mathbb{E}\left[\xi(A)\xi(B)\right]=\mathbb{E}\xi(A)\mathbb{E}\xi(B)+\text{Var}\xi(A\cap B).
 \]
 Using the approximation of $f$ by an increasing sequence of simple function, as in the construction of Lebesgue's integral, it follows from the Fubini-Tonelli theorem and the monotone convergence theorem, that
 \[
 \mathbb{E}\int_{[0,t]\times [0,s]}f(u,v)\ \xi(du)\xi(dv)=\int_{0}^{t}\theta f(u,u) \ du\\
 +\int_{0}^{t}\int_{0}^{s}\theta^{2}f(u,v)\ dudv.
 \]
 
 Since the desired relation only depends on the law of our random objects, we can assume without loss of generality that $\xi$ is defined on a probability space $\left(\Omega,\mathcal{F}, \mathbb{P} \right)$ and the family $\left(X^{(i)}_{s},\ s\in\mathbb{R}_{+}\right)_{i\geq 1}$ is defined on an other probability space $\left(\tilde{\Omega},\tilde{\mathcal{F}}, \tilde{\mathbb{P}} \right)$.
 Then, using a slight abuse of notation, we define $\xi$ on $\Omega\times\tilde{\Omega}$ by $\xi_{(\omega,\tilde{\omega})}=\xi_{\omega}$, and similarly for the family $X$.
 
 Then, by Fubini-Tonneli Theorem, with the notation $\xi_{\omega}^{v}=\xi_{\omega}\left([0,v] \right)$,
 \begin{multline*}
 \mathbb{E}\left[\int_{[0,t]\times[0,s]} X_{v}^{(\xi_{v})}X_{u}^{(\xi_{u})} \ \xi(du)\xi(dv)\right]=\int_{\Omega\times\tilde{\Omega}}\int_{[0,t]\times[0,s]} X_{v}^{(\xi_{\omega}^{v})}(\tilde{\omega})X_{u}^{(\xi_{\omega}^{u})}(\tilde{\omega}) \ \xi_{\omega}(du) \xi_{\omega}(dv)\ \mathbb{P}\otimes\tilde{\mathbb{P}}\left(d\omega,d\tilde{\omega}\right)\\
 =\int_{\Omega}\int_{[0,t]\times[0,s]}\left[\int_{\tilde{\Omega}} X_{v}^{(\xi_{\omega}^{v})}(\tilde{\omega})X_{u}^{(\xi_{\omega}^{u})}(\tilde{\omega})\tilde{\mathbb{P}}\left(d\tilde{\omega}\right)\right]  \ \xi_{\omega}(du) \xi_{\omega}(dv)\ \mathbb{P}(d\omega).
 \end{multline*}
 But since the $X^{(i)}$ are identically distributed and $\xi$ is a simple measure (purely atomic with mass one for each atom) we deduce that, if $u$ and $v$ are two atoms of $\xi_{\omega}$, $\xi_{\omega}^{v}=\xi_{\omega}^{u}$ if and only if $u=v$, which implies that
 \[
 \int_{\tilde{\Omega}} X_{v}^{(\xi_{\omega}^{v})}(\tilde{\omega})X_{u}^{(\xi_{\omega}^{u})}(\tilde{\omega})\tilde{\mathbb{P}}\left(d\tilde{\omega}\right)=
 \left\{
 \begin{array}{lr}
 \mathbb{E}X_{u}\mathbb{E}X_{v}, & u\neq v, \\ 
 \mathbb{E}X_{u}^{2}, & u=v,
 \end{array}
 \right. \quad \xi_{\omega}-a.e.
 \]
 The result follows readily, and the case with the indicator function of $Y$ is left to the reader.
 \end{proof}
\section{Statement of the theorem}
\label{sec:results}
 
The a.s.\ convergence stated in Section \ref{sec:models} suggests to study the second order properties of this convergence to get central limit theorems.
We recall that the Laplace distribution with zero mean and variance $\sigma^{2}$ is the probability distribution whose characteristic function is given by
\[
\lambda\in\mathbb{R}\mapsto\frac{1}{1+\frac{1}{2}\sigma^{2}\lambda^{2}}.
\]
It particular, it has a density given by
\[
x\in\mathbb{R}\mapsto\frac{1}{2\sigma}e^{-\frac{|x|}{\sigma}}.
\]

We denote this law by $\mathcal{L}\left(0,\sigma^{2} \right)$. We also recall that, if $G$ is a Gaussian random variable with zero mean and variance $\sigma^{2}$ and $\mathcal{E}$ is an exponential random variable with parameter $1$ independent of $G$, then $\sqrt{\mathcal{E}}G$ is Laplace $\mathcal{L}\left(0,\sigma^{2} \right)$. 

Before stating the main result of the paper, let us recall the law of large number for $N_{t}$.
 \begin{thm}
 \label{thm:ASCVNt2}
In the supercritical case, that is $b\mathbb{E}\left[V\right]>1$, there exists a random variable $\mathcal{E}$, such that
 \[
e^{-\alpha t}N_{t}\underset{t\to\infty}{\rightarrow}\frac{\mathcal{E}}{\psi'(\alpha)},\quad a.s. \text{ and in } L^{2}.
 \]
In particular, under $\mathbb{P}_{\infty}$, $\mathcal{E}$ is exponentially distributed with parameter one.
 \end{thm}
 \medskip
 
In this work we prove the following theorem on the second order properties of the above convergence.
\begin{thm}In the supercritical case, we have, under $\mathbb{P}_{\infty}$,
\label{thm:tclN}
 \[
e^{-\frac{\alpha}{2}t}\left(\psi'(\alpha)N_{t}-e^{\alpha t}\mathcal{E}\right)\overset{(d)}{\underset{t\to\infty}{\longrightarrow}}\mathcal{L}\left(0,2-\psi'(\alpha)\right).
\]

\end{thm}
The proof of this theorem is the subject of Section \ref{proof:th1}. Note that, according to \eqref{eq:laplaceDerivative}, we have
\[
2-\psi'(\alpha)=1+\int_{\mathbb{R}_{+}}ve^{-\alpha v }\ b\mathbb{P}_{V}(dv)>0.
\]
\section{Strategy of proof}
\label{sec:strategy}
Let $\left(G_{n}\right)_{n\geq 1}$ be a sequence of geometric random variables with respective parameter $\frac{1}{n}$, and $\left(X_{i}\right)_{i\geq 1}$ a $L^{2}$ family of i.i.d. random variables with zero mean independent of $\left(G_{n}\right)_{n\geq1}$.
It is easy to show that the characteristic function of
\begin{equation}
\label{eq:ConvGeom}
Z_{n}:=\frac{1}{\sqrt{n}}\sum_{i=1}^{G_{n}}X_{i},
\end{equation}
is given by
\begin{equation}
\label{eq:foncCara}
\mathbb{E}e^{i\lambda Z_{n}}=\frac{1+o_{n}(1)}{1+\lambda^{2}\mathbb{E}X_{1}^{2}+o_{n}(1)},
\end{equation}
from which we deduce that $Z_{n}$ converges in distribution to $\mathcal{L}(0,\mathbb{E}X_{1}^{2})$. 

If we suppose that the population counting process $N$ is a Yule Markov process, it clearly follows from the branching property that, for $s<t$,
\begin{equation}
\label{eq:Decomp}
N_{t}=\sum_{i=1}^{N_{s}}N^{i}_{t-s},
\end{equation}
where the family $\left(N^{i}_{t-s}\right)_{i\geq 1}$ is an i.i.d.\ sequence of random variables distributed as $N_{t-s}$ and independent of $N_{s}$. Moreover, since $N_{s}$ is geometrically distributed with parameter $e^{-\alpha s}$, taking the renormalized limit leads to,
\[
\lim\limits_{t\to\infty}e^{-\alpha t}N_{t}=:\mathcal{E}=e^{-\alpha s}\sum_{i=1}^{N_{s}}\mathcal{E}_{i},
\]
where $\mathcal{E}_{1},\dots,\mathcal{E}_{N_{s}}$ is an i.i.d.\ family of exponential random variables with parameter one, and independent of $N_{s}$. Hence,
\[
N_{t}-e^{\alpha t}\mathcal{E}=\sum_{i=1}^{N_{s}}\left(N^{i}_{t-s}-e^{\alpha (t-s)}\mathcal{E}_{i} \right),
\] is a geometric sum of centered i.i.d.\ random variables. This remark and \eqref{eq:ConvGeom} suggest the desired CLT in the Yule case.
\begin{rem}
\label{rem:indep}
Let $N$ be a integer valued random variable. In the sequel we say that a random vector with random size $\left(X_{i}\right)_{1\leq i\leq N}$ form an i.i.d.\ family of random variables independent of $N$, if and only if
\[
\left(X_{1},\dots,X_{N}\right)\overset{d}{=}\left(\tilde{X_{1}},\dots,\tilde{X}_{N}\right),
\] 
where $\left(\tilde{X}_{i}\right)_{i\geq 1}$ is a sequence of i.i.d.\ random variables distributed as $X_{1}$ independent of $N$.
\end{rem}
 However, in the general case, we need to overcome some important difficulties. First of all, equation \eqref{eq:Decomp} is wrong in general. Nevertheless, a much weaker version of \eqref{eq:Decomp} can be obtained in the general case. To make this clear, if $u<t$ are two positive real numbers, then the number of alive individuals at time $t$ is the sum of the contributions of each subtrees $\mathbb{T}\left(O_{i}\right)$ induced by each alive individuals at time $u$ (see Figure \ref{fig:subtrees}). Provided there are individuals alive at time $u$, we denote by $\left(O_{i}\right)_{1\leq i\leq N_{u}}$ the residual lifetimes (see Figure \ref{fig:subtrees}) of the alive individuals at time $u$ indexed using that the $i$th individual is the $i$th individual visited by the contour process. Hence,
 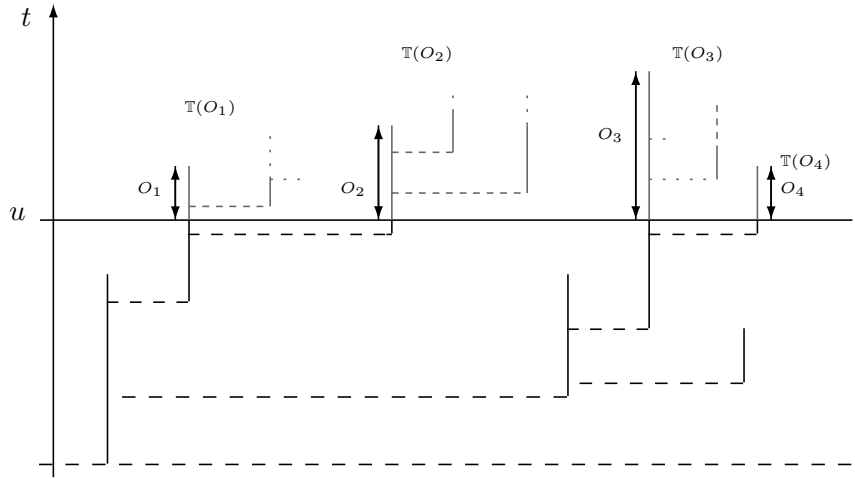
\begin{figure}[ht]
 \unitlength 2mm 
 \linethickness{0.4pt}
 
 \unitlength 1.8mm 
 \linethickness{0.6pt}
 \ifx\plotpoint\undefined\newsavebox{\plotpoint}\fi 
 \begin{picture}(68,40)(-10,0)
 \put(11,24){\makebox{ $u$}}
 \put(14,24){\line(1,0){60}}
 \put(15,5){\vector(0,1){35}}
 \put(13.93,5.93){\line(1,0){.9836}}
 \put(15.897,5.93){\line(1,0){.9836}}
 \put(17.864,5.93){\line(1,0){.9836}}
 \put(19.831,5.93){\line(1,0){.9836}}
 \put(21.799,5.93){\line(1,0){.9836}}
 \put(23.766,5.93){\line(1,0){.9836}}
 \put(25.733,5.93){\line(1,0){.9836}}
 \put(27.7,5.93){\line(1,0){.9836}}
 
 \put(29.667,5.93){\line(1,0){.9836}}
 \put(31.635,5.93){\line(1,0){.9836}}
 \put(33.602,5.93){\line(1,0){.9836}}
 \put(35.569,5.93){\line(1,0){.9836}}
 \put(37.536,5.93){\line(1,0){.9836}}
 \put(39.503,5.93){\line(1,0){.9836}}
 \put(41.471,5.93){\line(1,0){.9836}}
 \put(43.438,5.93){\line(1,0){.9836}}
 \put(45.405,5.93){\line(1,0){.9836}}
 \put(47.372,5.93){\line(1,0){.9836}}
 
 \put(49.34,5.93){\line(1,0){.9836}}
 \put(51.307,5.93){\line(1,0){.9836}}
 \put(53.274,5.93){\line(1,0){.9836}}
 \put(55.241,5.93){\line(1,0){.9836}}
 \put(57.208,5.93){\line(1,0){.9836}}
 \put(59.176,5.93){\line(1,0){.9836}}
 \put(61.143,5.93){\line(1,0){.9836}}
 \put(63.11,5.93){\line(1,0){.9836}}
 \put(65.077,5.93){\line(1,0){.9836}}
 \put(67.044,5.93){\line(1,0){.9836}}
 \put(69.012,5.93){\line(1,0){.9836}}
 \put(70.979,5.93){\line(1,0){.9836}}
 \put(72.946,5.93){\line(1,0){.9836}}
 
 \put(19,20){\line(0,-1){14}}
 \color{gray1}
 \put(25,28){\line(0,-1){4}}
 \multiput(25,25)(1,0){6}{\line(1,0){0.5}}
 \put(31,25){\line(0,1){2}}
 \color{gray1}
 \multiput(31,27)(0,1){4}{\line(0,1){0.2}}
 \multiput(31,27)(1,0){3}{\line(1,0){0.2}}
 \color{black}
 \put(24,28){\vector(0,-1){4}}
 \put(24,24){\vector(0,1){4}}
 \put(20.5,26){\makebox{ {\tiny $O_{1}$}}}
   \put(24,32){\makebox{ {\tiny $\mathbb{T}(O_{1})$}}}

 \put(25,24){\line(0,-1){6}}
 \color{black}
 \put(40,23){\line(0,1){1}}

 \color{gray1}
 
 \put(40,24){\line(0,1){7}}
 \multiput(40,29)(1,0){5}{\line(1,0){0.5}}
 \put(44.5,29){\line(0,1){3}}
 \multiput(44.5,32)(0,1){2}{\line(0,1){0.2}}
 \multiput(40,26)(1,0){10}{\line(1,0){0.5}}
 \put(50,26){\line(0,1){5}}
 \multiput(50,31)(0,1){3}{\line(0,1){0.2}}
 \color{black}
 \put(39,31){\vector(0,-1){7}}
 \put(39,24){\vector(0,1){7}}
 
  \put(35.5,26){\makebox{ {\tiny $O_{2}$}}}
   \put(40,36){\makebox{ {\tiny $\mathbb{T}(O_{2})$}}}
 \put(53,11){\line(0,1){9}}
 \color{black}
 \put(59,16){\line(0,1){8}}
 \color{gray1}
 \multiput(59,27)(1,0){5}{\line(1,0){0.2}}
 \multiput(59,30)(1,0){2}{\line(1,0){0.2}}
 \put(64,27){\line(0,1){2}}
 \multiput(64,29)(0,1){4}{\line(0,1){0.5}}
 
 \put(59,24){\line(0,1){11}}
 \color{black}
   \put(60,36){\makebox{ {\tiny $\mathbb{T}(O_{3})$}}}
 \put(54.5,30){\makebox{ {\tiny $O_{3}$}}}
 \put(58,35){\vector(0,-1){11}}
 \put(58,24){\vector(0,1){11}}
 \put(24.93,17.93){\line(-1,0){.8571}}
 \put(23.215,17.93){\line(-1,0){.8571}}
 \put(21.501,17.93){\line(-1,0){.8571}}
 \put(19.787,17.93){\line(-1,0){.8571}}
 
 \put(30.93,22.93){\line(-1,0){.8571}}
 \put(29.215,22.93){\line(-1,0){.8571}}
 \put(27.501,22.93){\line(-1,0){.8571}}
 \put(25.787,22.93){\line(-1,0){.8571}}
 \put(32.645,22.93){\line(-1,0){.8571}}
 \put(34.36,22.93){\line(-1,0){.8571}}
 \put(36.075,22.93){\line(-1,0){.8571}}
 \put(37.79,22.93){\line(-1,0){.8571}}
 \put(39.505,22.93){\line(-1,0){.8571}}
 \put(40,22.93){\line(-1,0){.4}}
 \color{black}
 \put(58.93,15.93){\line(-1,0){.8571}}
 \put(57.215,15.93){\line(-1,0){.8571}}
 \put(55.501,15.93){\line(-1,0){.8571}}
 \put(53.787,15.93){\line(-1,0){.8571}}
 \put(52.93,10.93){\line(-1,0){.9962}}
 \put(50.937,10.93){\line(-1,0){.9962}}
 \put(48.945,10.93){\line(-1,0){.9962}}
 \put(46.953,10.93){\line(-1,0){.9962}}
 \put(44.96,10.93){\line(-1,0){.9962}}
 \put(42.968,10.93){\line(-1,0){.9962}}
 \put(40.976,10.93){\line(-1,0){.9962}}
 \put(38.983,10.93){\line(-1,0){.9962}}
 \put(36.991,10.93){\line(-1,0){.9962}}
 \put(34.999,10.93){\line(-1,0){.9962}}
 \put(33.006,10.93){\line(-1,0){.9962}}
 \put(31.014,10.93){\line(-1,0){.9962}}
 \put(29.021,10.93){\line(-1,0){.9962}}
 \put(27.029,10.93){\line(-1,0){.9962}}
 \put(25.037,10.93){\line(-1,0){.9962}}
 \put(23.044,10.93){\line(-1,0){.9962}}
 \put(21.052,10.93){\line(-1,0){.9962}}
 
 \put(67,23){\line(0,1){1}}
 \color{gray1}
 
 \put(67,24){\line(0,1){4}}
 \color{black}
   \put(68,28){\makebox{ {\tiny $\mathbb{T}(O_{4})$}}}
 \put(68,26){\makebox{ {\tiny $O_{4}$}}}
 \put(68,28){\vector(0,-1){4}}
 \put(68,24){\vector(0,1){4}}
 \put(66,12){\line(0,1){4}}
 
 \put(65.93,11.93){\line(-1,0){.9286}}
 \put(64.073,11.93){\line(-1,0){.9286}}
 \put(62.215,11.93){\line(-1,0){.9286}}
 \put(60.358,11.93){\line(-1,0){.9286}}
 \put(58.501,11.93){\line(-1,0){.9286}}
 \put(56.644,11.93){\line(-1,0){.9286}}
 \put(54.787,11.93){\line(-1,0){.9286}}
 \put(66.93,22.93){\line(-1,0){.8889}}
 \put(65.152,22.93){\line(-1,0){.8889}}
 \put(63.374,22.93){\line(-1,0){.8889}}
 \put(61.596,22.93){\line(-1,0){.8889}}
 \put(59.819,22.93){\line(-1,0){.8889}}
 \put(13,39){\makebox(0,0)[cc]{$t$}}
 
 \end{picture}
 \caption{ Residual lifetimes with subtrees associated to living individuals at time $u$.}
 \label{fig:subtrees}
 \end{figure}
 \begin{equation}
 \label{eq:decompNt}
 N_{t}=\sum_{i=1}^{N_{u}}N^{i}_{t-u}\left(O_{i} \right),
 \end{equation}
where $\left(N^{i}_{t-u}\left(O_{i} \right)\right)_{i\leq N_{u}}$ denote the population counting processes of the subtrees $\mathbb{T}(O_{i})$ induced by each individual. The notation refers to the fact that each subtree has the law of a standard splitting tree with the only difference that the lifelength of the root is given by $O_{i}$. More precisly, we define, for all $i\geq 1$ and $o\in\mathbb{R}_{+}$, $N^{i}_{t-u}(o)$ the population counting process of the splitting tree constructed from the same random objects as the $i$th subtree of Figure \ref{fig:subtrees}, where the life duration of the first individual is equal to $o$. Hence, from the independence properties between each individuals, $\left(N^{i}_{t-u}\left(o\right), \ t\geq u,o\geq0\right)_{i\geq 1}$ is a family of independent processes, independent of $\left(O_{i}\right)_{1\leq i\leq N_{u}}$, and $\left(N_{t-u}^{i}(o), t\geq u \right)$ has the law of the population counting process of a splitting tree but where the lifespan of the ancestor is $o$.  Note that the lifespans of the other individuals are still distributed as $V$. From the discussion above, it follows that the family of processes $\left(N^{i}_{t-u}\left(O_{i}\right),\ t\geq u \right)_{1\leq i\leq N_{u}}$ are dependent only through the residual lifetimes $\left(O_{i} \right)_{1\leq i\leq N_{u}}$ and the law of $\left(N_{t}\left(O_{i}\right),\ t\in\mathbb{R}_{+}\right)$ under $\mathbb{P}_{u}$ is the law of standard population counting process of splitting tree where the lifespan of the root is distributed as $O_{i}$ under $\mathbb{P}_{u}$.

 Unfortunately, the computation of \eqref{eq:foncCara} does not apply to \eqref{eq:decompNt}.
This issue is solved by the following lemma, whose proof is very similar to one of Proposition 5.5 of \cite{L10}.
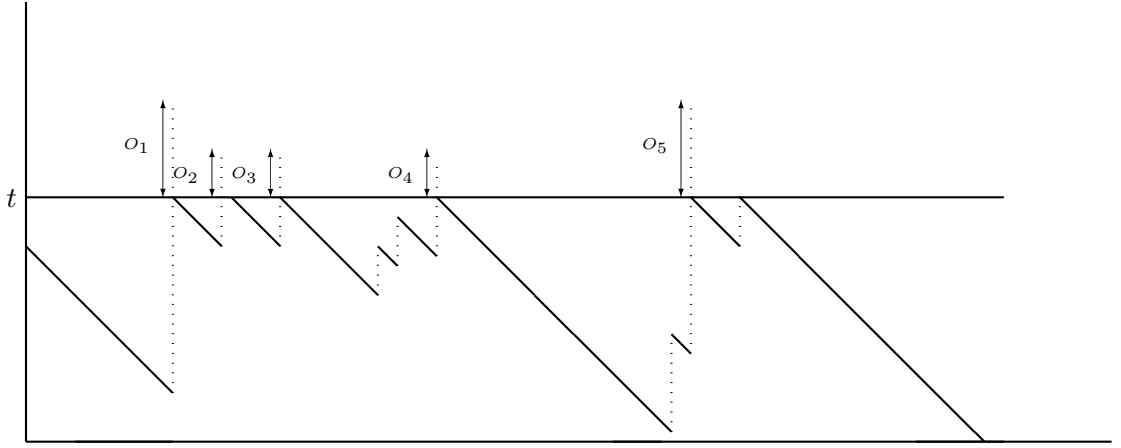
\begin{figure}[ht]
\unitlength  2mm 
\linethickness{0.8pt}
\center{
\unitlength 1.3mm
\begin{picture}(100,60)(0,0)
\put(0,4){
\put(0,24){\makebox{$t$}}
\put(2,0){
\put(0,0){\line(1,0){100}}
\put(0,25){\line(1,0){100}}
\put(0,0){\line(0,1){45}}

%
%
%


%

\put(5,0){\line(1,0){10}}
\put(60,0){\line(1,0){5}}
\put(91,0){\line(1,0){20}}
%
%
%

\put(0,20){\line(1,-1){15}}
\linethickness{0.2pt}
\put(14,25){\vector(0,1){10}}
\put(14,35){\vector(0,-1){10}}
\linethickness{0.8pt}
\put(10,30){\makebox{$\scriptscriptstyle{O_{1}}$}}
\multiput(15,5)(0,1){30}{\line(0,1){0.1}}

\put(15,25){\line(1,-1){5}}
\multiput(20,20)(0,1){10}{\line(0,1){0.1}}
\linethickness{0.2pt}
\put(19,25){\vector(0,1){5}}
\put(19,30){\vector(0,-1){5}}
\linethickness{0.8pt}
\put(15,27){\makebox{$\scriptscriptstyle{O_{2}}$}}
\linethickness{0.8pt}
\color{black}
\put(20,25){\line(1,-1){5}}
\multiput(25,20)(0,1){10}{\line(0,1){0.1}}
\put(25,25){\line(1,-1){10}}
\linethickness{0.2pt}
\put(24,25){\vector(0,1){5}}
\put(24,30){\vector(0,-1){5}}
\linethickness{0.8pt}
\put(20,27){\makebox{$\scriptscriptstyle{O_{3}}$}}

\multiput(35,15)(0,1){5}{\line(0,1){0.1}}
\put(35,20){\line(1,-1){2}}


\multiput(37,18)(0,1){5}{\line(0,1){0.1}}
\put(37,23){\line(1,-1){4}}
\linethickness{0.2pt}
\put(40,25){\vector(0,1){5}}
\put(40,30){\vector(0,-1){5}}
\linethickness{0.8pt}
\put(36,27){\makebox{$\scriptscriptstyle{O_{4}}$}}
\multiput(41,19)(0,1){10}{\line(0,1){0.1}}
\put(41,25){\line(1,-1){10}}
\put(51,15){\line(1,-1){1}}
\put(52,14){\line(1,-1){3}}
\put(55,11){\line(1,-1){10}}
\multiput(65,1)(0,1){10}{\line(0,1){0.1}}
\put(65,11){\line(1,-1){2}}
\linethickness{0.2pt}
\put(66,25){\vector(0,1){10}}
\put(66,35){\vector(0,-1){10}}
\linethickness{0.8pt}
\put(62,30){\makebox{$\scriptscriptstyle{O_{5}}$}}
\multiput(67,10)(0,1){25}{\line(0,1){0.1}}
\put(67,25){\line(1,-1){5}}
\multiput(72,20)(0,1){6}{\line(0,1){0.1}}
\put(72,25){\line(1,-1){6}}
\put(78,19){\line(1,-1){10}}
\put(88,9){\line(1,-1){8}}
\put(96,1){\line(1,-1){1}}
\linethickness{0.2pt}

}}
\end{picture}}
\caption{Reflected JCCP with overshoot over $t$. Independence is provided by the Markov property. }

\end{figure}
\begin{lem}
\label{lem:residual}Let $u$ in $\mathbb{R}_{+}$, we denote by $O_{i}$ for $i$ an integer between $1$ and $N_{u}$ the residual lifetime of the $i$th individuals alive at time $u$. Then under $\mathbb{P}_{u}$, the family $\left(O_{i},\ i\in\llbracket1,N_{u}\rrbracket\right)$ form a family of independent random variables, independent of $N_{u}$, and, expect $O_{1}$, having the same distribution, given by, for $2\leq i\leq N_{t}$,
\begin{equation}
\label{eq:loiOversh}
\mathbb{P}_{u}(O_{i}\in dx)=\int_{\mathbb{R}_{+}}\ \frac{W(u-y)}{W(u)-1}b \mathbb{P}\left(V-y\in dx \right)\ dy.
\end{equation}
Moreover, it follows that the family $\left(N_{s}(O_{i}),s\in\mathbb{R}_{+}\right)_{1\leq i\leq N_{u}}$ is an independent family of process, i.i.d.\ for $i\geq 2$, and independent of $N_{u}$.
\end{lem}
\begin{proof}
Let $\left(Y^{(i)}\right)_{0\leq i\leq N_{u}}$ a family of independent L\'evy processes with Laplace exponent
\[
\psi(x)=x-\int_{(0,\infty]}\left(1-e^{-rx}\right)\Lambda(dr), \ \ x\in\mathbb{R}_{+},
\]
conditioned to hit $(u,\infty)$ before hitting $0$, for $i\in\left\{0,\dots,N_{u}-1 \right\}$, and conditioned to hit $0$ first for $i=N_{u}$.
We also assume that,
\[
Y^{(0)}_{0}=u\wedge V,
\]
and
\[
Y^{(i)}_{0}=u,\quad i\in\left\{1,\dots,N_{u} \right\}.
\]
Now, denote by $\tau_{i}$ the exit time of the $i$th process out of $(0,u)$ and
\[
T_{n}=\sum_{i=0}^{n-1}\tau_{i},\quad  n\in\left\{0,\dots,N_{u}+1 \right\}.
\] 
Then, the process defined, for all $s$, by
\[
Y_{s}=\sum_{i=0}^{N_{u}}Y^{(i)}_{s-T_{i}}\mathds{1}_{T_{i}\leq s< T_{i+1}},
\]
has the law of the contour process of a splitting tree cut above $u$. Moreover, the quantity $Y_{\tau_{i}}-Y_{\tau_{i}-}$ is the lifetime of the $i$th alive individual at time $t$.
The family of residual lifetimes $\left(O_{i}\right)_{1\leq i\leq N_{u}}$ has then the same distribution as the sequence of the overshoots of the $Y$ above $u$.
Thus, the Markov property ensures us that $\left(O_{i},\ i\in\llbracket2,N_{u}\rrbracket \right)$ is an i.i.d.\ family of random variables. The Markov property also ensures that $O_{1}$ is independent of the other $O_{i}$'s.

It remains to derive the law of $O_{i}$. Let $Y$ be a L\'evy process with Laplace exponent $\psi$. We denote by $\tau^{+}_{u}$ the time of first passage of $-Y$ above $u$ and $\tau^{-}_{0}$ the time of first passage of $-Y$ below $0$. Then, for all $i\geq2$,

\[
\mathbb{P}_{u}\left(O_{i}\in dx \right)=\mathbb{P}_{0}\left(-Y_{\tau^{-}_{0}}\in dx \mid \tau^{-}_{0}<\tau_{u}^{+} \right).
\]
On the other hand, Theorem 8.7 of \cite{Kyp} gives for any measurable subsets $A\subset [0,u]$, $B\subset(0,-\infty)$,
\[
\mathbb{P}_{0}\left(-Y_{\tau^{-}_{0}}\in B, -Y_{\tau_{0}^{-}-}\in A \right)=\int_{A}\mathbb{P}_{-V}\left(B-y\right)\frac{W(u-y)}{W(u)}dy.
\]
The result follows easily from
\[
\mathbb{P}\left(\tau^{-}_{0}<\tau_{u}^{+}\right)=1-\frac{1}{W(u)}.\qedhere
\]
\end{proof}

\begin{rem} It is important to note that the law of the residual lifetimes of the individuals considered above depends on the particular time $u$ we choose to cut the tree.
That is why, in the sequel, we may denote $O^{(u)}_{i}$ for $O_{i}$ when we want to underline the dependence in time of the law of the residual lifetimes.
\end{rem}
In addition, as suggested by \eqref{eq:foncCara}, we need to compute the expected quadratic error in the convergence of $N_{t}$,
\[
\mathbb{E}\left[\left(\psi'(\alpha)N_{t}-e^{\alpha t}\mathcal{E} \right)^{2}\right],
\]
 which implies to compute
$
\mathbb{E}N_{t}\mathcal{E}.
$

Although, this moment is easy to obtain in the Markovian case, the method does not extend easily to the general case.
One idea  is to characterize it as a solution of a renewal equation in the spirit of the theory of general CMJ processes. 

To make this, we use the renewal structure of a splitting tree: the splitting trees can be constructed (see \cite{L10}) by grafting i.i.d.\ splitting tree on a branch (a tree with a single individual) of length $V_{\emptyset}$ distributed as $V$. Therefore, there exists a family $\left(N^{(i)}_{t},\  t\in\mathbb{R}_{+} \right)_{i\geq 1}$ of i.i.d.\ population counting processes with the same law as $\left(N_{t},\ t\in\mathbb{R}_{+}\right)$, and a Poisson random measure $\xi$ on $\mathbb{R}_{+}$ with intensity $b\, da$ such that
\begin{equation}
\label{eq:decomposition}
N_{t}=\int_{[0,t]}N^{(\xi_{u})}_{t-u}\mathds{1}_{V_{\emptyset}>u}\ \xi(du)+\mathds{1}_{V_{\emptyset}> t}, \quad a.s.,
\end{equation}
where $\xi_{u}=\xi\left([0,u]\right)$.


Another difficulty comes from the fact that unlike \eqref{eq:ConvGeom}, the quantities summed in \eqref{eq:decompNt} are time-dependent, which requires a careful analysis of the asymptotic behaviour of their moments. 

The calculus and the asymptotic analysis of these moments is made in Section \ref{ssec:moment1}:
In Lemma \ref{lem:joinMom}, we compute $\mathbb{E}N_{t}\mathcal{E}$, and then with Lemmas \ref{lem:quadeR} and \ref{lem:NtBound}, we study the asymptotic behaviour of the error of order 2 and 3 respectively.
Section \ref{ssec:Ntarbi} is devoted to the study of the same questions for the population counting processes of the subtrees described in Figure \ref{fig:subtrees} (when the lifetime of the root is not distributed as $V$). Finally, Section \ref{sec:endProof} is devoted to the proof of Theorem \ref{thm:tclN}.

One of the difficulties in studying the behaviour of the moments is to get better estimates on the scale function $W$ than those of Lemma \ref{lem: asyComp}. This is the subject of the next section.
\label{sec:strat}
\section{Precise estimates on $W$ using L\'evy processes}
Before stating and proving the result of this section, we need to recall some facts about L\'evy processes. We follow the presentation of \cite{Kyp}. First, we recall that the law of a spectrally positive L\'evy process $\left(Y_{t},\ t\in\mathbb{R}_{+} \right)$ is uniquely characterized by its Laplace exponent $\psi$,
\[
\psi_{Y}\left(\lambda \right)=\log\mathbb{E}\left[e^{-\lambda Y_{1}} \right], \ \lambda\in\mathbb{R}_{+},
\]
which in our case take the form of \eqref{eq:laplace}:
\[
\psi_{Y}(\lambda)=x-\int_{(0,\infty]}\left(1-e^{-rx}\right)b\mathbb{P}_{V}(dr), \ \ \lambda\in\mathbb{R}_{+}.
\]
In this section, we suppose that $Y_{0}=0$.
For a such L\'evy process, $0$ is irregular for $(0,\infty)$ and in this case the local time at the maximum $\left(L_{t}, \ t\in\mathbb{R}\right)$ can be defined as
\[
L_{t}=\sum_{i=0}^{n_{t}}e^{i}, \quad t\in\mathbb{R}_{+},
\]
where $\left(e^{i}\right)_{i\geq 0}$ is a family of i.i.d.\ exponential random variables with parameter $1$, and 
\[
n_{t}:=\text{Card}\{0<s\leq t \mid Y_{s}=\sup_{u\leq s }Y_{u} \},
\]
is the number of times $Y$ reaches its maximum up to time $t$.
Finally, the ascending ladder process associated to $Y$ is defined as
\[
H_{t}=\sup_{s\leq L^{-1}_{t}}Y_{s},\quad t\in\mathbb{R}_{+},
\]
where $\left(L_{t}^{-1}, t\in\mathbb{R}_{+}\right)$ is the right-inverse of $L$.
It is known that $H$ is a subordinator whose values are the successive new maxima of $Y$.
Conversely, in our case, the process $\left(\inf_{s\leq t}Y_{s},\ t\in\mathbb{R}_{+} \right)$ is a local time at the minimum, denoted $\left(\hat{L}_{t}, \ t\in\mathbb{R}_{+}\right)$. The descending ladder process $\hat{H}$ is then defined from $\hat{L}$ as $H$ was defined from $L$.

We can now state, the celebrated Wiener-Hopf factorization  which allows us to connect the characteristic exponent $\psi_{Y}$ of $Y$ with the characteristic exponents of the bivariate L\'evy processes $\left(\left(L_{t},H_{t}\right),\ t\in\mathbb{R}_{+} \right)$ and $\left(\left(\hat{L}_{t},\hat{H}_{t}\right),\ t\in\mathbb{R}_{+} \right)$, respectively denoted by $\kappa$ and $\hat{\kappa}$. In our particular case, where $Y$ is spectrally negative, we have
\[
\left\{
\begin{array}{lr}
\kappa(\alpha,\beta)=\frac{\alpha-\psi_{Y}(\beta)}{\phi_{Y}(\alpha)-\beta}, & \alpha,\beta\in\mathbb{R}_{+}, \\ 
 \hat{\kappa}(\alpha,\beta)=\phi_{Y}(\alpha)+\beta, & \alpha,\beta\in\mathbb{R}_{+},
\end{array} 
\right.
\]
where $\phi_{Y}$ is the right-inverse of $\psi_{Y}$.
Taking $\alpha=0$ allows us to recover the Laplace exponent $\psi_{H}$ of $H$ from which we obtain the relation,
\begin{equation}
\label{eq:wienerHopf}
\psi_{Y}(\lambda)=\left(\lambda-\phi_{Y}(0) \right)\psi_{H}(\lambda).
\end{equation}
We have now all the notation to state and prove the main result of this section.
\begin{prop}[Behavior of $W$]
\label{lem:WComp}
There exists a positive non-increasing c\`adl\`ag function $F$ such that
\[
W(t)=\frac{e^{\alpha t}}{\psi'(\alpha)}-e^{\alpha t}F(t),\quad t\geq 0,
\]
and
\[
\lim\limits_{t\to \infty}e^{\alpha t}F(t)=
\begin{cases}
\frac{1}{b\mathbb{E}V-1} & \mbox{if}\ \mathbb{E}V<\infty, \\ 
0 & \mbox{otherwise.}
\end{cases}
\]
\end{prop}
\begin{proof}
\label{lem: compScale}
Let $Y^{\sharp}$ be a spectrally negative L\'evy process with Laplace exponent given by
\[
\psi^{\sharp}(\lambda)=\lambda-\int_{\mathbb{R}_{+}}\left(1-e^{-\lambda x} \right)e^{-\alpha x}b\ \mathbb{P}_{V}(dx).
\]
It is known that $Y^{\sharp}$ has the law of the contour process of the supercritical splitting tree with lifespan measure $\mathbb{P}_{V}$ conditioned to extinction (see \cite{L10}). In this case the largest root of $\psi^{\sharp}$ is zero, meaning that the process $Y^{\sharp}$ does not go to infinity and that $\phi_{Y^{\sharp}}(0)=0$. Elementary manipulations on Laplace transform show that the scale function $W^{\sharp}$ of $Y^{\sharp}$ is related to $W$ by
\[
W^{\sharp}(t)=e^{-\alpha t}W(t),\quad t\in\mathbb{R}_{+}.
\]

Let $H^{\sharp}$ be the ascending ladder subordinator associated to the L\'evy process $Y^{\sharp}$.
In the case where $\phi_{Y^{\sharp}}(0)=0$, and in this case only, the scale function $W^{\sharp}$ can be rewritten as (see \cite{Kyp} or use Laplace transform),
\begin{equation}
\label{eq:scalePotMes}
W^{\sharp}(t)=\int_{0}^{\infty}\mathbb{P}\left(H^{\sharp}_{x}\leq t \right)dx.
\end{equation}
In other words, if we denote by $U$ the potential measure of $H^{\sharp}$,
\[
W^{\sharp}(t)=U[0,t].
\]
Now, it is easily seen from \eqref{eq:wienerHopf} that the Laplace exponent $\psi_{H^{\sharp}}$ of $H^{\sharp}$ takes the form,
\[
\psi_{H^{\sharp}}\left(\lambda \right)=\psi'(\alpha)-\int_{[0,\infty]}\left(1-e^{-\lambda r}\right)\Upsilon(dr),
\]
where
\[
\Upsilon(dr)=\int_{(r,\infty)}e^{-\alpha v}b\mathbb{P}_{V}(dv)dr=\mathbb{E}\left[e^{-\alpha V}\mathds{1}_{V>r} \right]bdr.
\]
Moreover,
\[
\Upsilon(\mathbb{R}_{+})=1-\psi'(\alpha),
\]
which mean that $H^{\sharp}$ is a compound Poisson process with jump rate $1-\psi'(\alpha)$, jump distribution $J(dr):=\frac{\mathbb{E}\left[e^{-\alpha V}\mathds{1}_{V>r} \right]}{1-\psi'(\alpha)}dr$, and killed at rate $\psi'(\alpha)$. It is well known (or elementary by conditioning on the number of jumps at time $x$), that
\[
\mathbb{P}_{H^{\sharp}_{x}}(dt)=e^{-\psi'(\alpha) x}\sum_{k\geq 0}e^{-\left(1-\psi'(\alpha) \right)x}\frac{\left(\left(1-\psi'(\alpha) \right)x\right)^{k}}{k!}J^{\star k}(dt).
\]
Some calculations now lead to,
\[
U(dx)=\sum_{k\geq 0}\Upsilon^{\star k}(dx).
\]
From this point, since $\Upsilon$ is a sub-probability, $U(x):=U[0,x]$ satisfies the following defective renewal equation,
\[
U(x)=\int_{\mathbb{R}_{+}}U(x-u)\Upsilon(du)+\mathds{1}_{\mathbb{R}_{+}}(x).
\]
Finally, since
\[
\int_{\mathbb{R}_{+}}e^{\alpha x}\Upsilon(dx)=1,
\]
and since, from Lemma \ref{lem:DRI},
\[
t\rightarrow U(t,\infty),
\]
is clearly a directly Riemann integrable function as a positive decreasing integrable function.
Hence, as suggested in Remark \ref{rem:1},
\[
e^{\alpha x}\left(U(\mathbb{R}_{+})-U(x)\right)\underset{x\to\infty}{\longrightarrow}\frac{1}{\alpha \mu},
\]
with 
\[
\mu=\int_{\mathbb{R}_{+}}re^{\alpha r}\Upsilon(dr)=\frac{1}{\alpha}\left(b\mathbb{E}V-1 \right),
\]
if $V$ is integrable. In the case where $V$ is not integrable, the limit is $0$.

To end the proof, note using relation \eqref{eq:scalePotMes} and the fact that $H^{\sharp}$ is killed at rate $\psi'(\alpha)$ that,
\[
W^{\sharp}(t)=\frac{1}{\psi'(\alpha)}-U(t,\infty).\qedhere
\]
\end{proof}
\section{Proof of Theorem \ref{thm:tclN}}
\label{proof:th1}
We begin the proof of Theorem \ref{thm:tclN} by computing moments, and analysing their asymptotic behaviours. A first part is devoted to the case of a splitting tree where the lifetime of the root is distributed as $V$ whereas a second part study the case where the lifespan of the root is arbitrary (for instance, as the subtrees described by Figure \ref{fig:subtrees}).
\subsection{Preliminary moments estimates}
This section is devoted to the calculus of the expectation of $\left(N_{t}-e^{\alpha t}\mathcal{E}\right)^{2}$. We start with the simple case where the initial individual has life-length distributed as $V$. Secondly, we study the asymptotic behavior of these moments. In Subsection \ref{ssec:Ntarbi}, we prove similar result for arbitrary initial distributions.

The expectations above are given with respect to $\mathbb{P}$, however since $N_{t}$ and $\mathcal{E}$ vanish on the extinction event, we can easily recover the results with respect to $\mathbb{P}_{t}$ by using  \eqref{eq:Prnonex} and \eqref{eq:probaNoCond} (see Corollary \ref{cor:withCond}).

\subsubsection{Case $V_{\emptyset}\overset{\mathcal{L}}{=}V$}
\label{ssec:moment1}
We start with the computation of $\mathbb{E}N_{t}\mathcal{E}$.

\begin{lem}[Join moment of $\mathcal{E}$ and $N_{t}$]
\label{lem:joinMom}

The function $t\to\mathbb{E}\left[N_{t}\mathcal{E}\right]$ is the unique solution bounded on finite intervals of the renewal equation,
\begin{align}
f(t)=&\int_{\mathbb{R}_{+}}f(t-u)be^{-\alpha u}\mathbb{P}\left(V>u\right)du\nonumber
\\&+\alpha b\mathbb{E}\left[N_{\cdot}\right]\star \left( \int_{\mathbb{R}_{+}}e^{-\alpha v}\mathbb{P}\left(V>\cdot,V>v\right)dv\right)(t)\nonumber \\&+\alpha\int_{\mathbb{R}_{+}}e^{-\alpha v}\mathbb{P}\left(V>t,V>v\right)dv\label{eq:renewNt},
\end{align}
and its solution is given by
\[
\left(1+\frac{\alpha}{b}-e^{-\alpha t}\right)W(t)-\left(1-e^{-\alpha t} \right)W\star\mathbb{P}_{V}(t).
\]
\end{lem}
\begin{proof}
As explained in Section \ref{sec:strategy},
\[
N_{t}=\int_{[0,t]}N^{(\xi_{u})}_{t-u}\mathds{1}_{V_{\emptyset}>u}\ \xi(du)+\mathds{1}_{V_{\emptyset}>t},
\] 
where $\xi$ a Poisson point process with rate $b$ on the real line, $\left(N^{(i)}\right)_{i\geq 1}$ is a family of independent CMJ processes with the same law as $N$ and $V_{\emptyset}$ is the lifespan of the root. Moreover, the three objects $N^{(u)}$, $\xi$ and $V_{\emptyset}$ are independent.

It follows that, for $s>t$
\begin{multline*}
N_{t}N_{s}=\int_{[0,t]\times[0,s]}N^{(\xi_{u})}_{t-u}N^{(\xi_{v})}_{s-v}\mathds{1}_{V_{\emptyset}>u}\mathds{1}_{V_{\emptyset}>v}\ \xi(du)\xi(dv)\\
+\int_{[0,t]}N^{(\xi_{u})}_{t-u}\mathds{1}_{V_{\emptyset}>u}\ \xi(du)\mathds{1}_{V_{\emptyset}>s}+\int_{[0,s]}N^{(\xi_{u})}_{s-u}\mathds{1}_{V_{\emptyset}>u}\ \xi(du)\mathds{1}_{V_{\emptyset}>t}
+\mathds{1}_{V_{\emptyset}>t}\mathds{1}_{V_{\emptyset}>s},
\end{multline*}
and, using Lemma \ref{lem:stocInt},
\begin{align*}
\mathbb{E}N_{t}N_{s}=&\int_{[0,t]}b\mathbb{E}\left[N_{t-u}N_{s-u}\right]\ \mathbb{P}\left(V>u\right)\ du\\
&+\int_{[0,t]\times[0,s]}b^{2}\mathbb{E}\left[N_{t-u}\right]\mathbb{E}\left[N_{s-v}\right]\mathbb{P}\left(V>u,V>v\right)\ du\ dv\\
&+\mathbb{P}\left(V>s\right)\int_{[0,t]}b\mathbb{E}\left[N_{t-u}\right]\ du+\int_{[0,s]}b\mathbb{E}\left[N_{s-u}\right]\mathbb{P}\left(V>u,V>t\right)\ du+\mathbb{P}\left(V>s\right).
\end{align*}
Then, thanks to the estimate
$W(t)=\mathcal{O}\left(e^{\alpha t}\right)$ (see Lemma \ref{lem: asyComp} or \ref{lem:WComp})
and the $L^{1}$ convergence of $W(s)^{-1}N_{t}N_{s}$ to $N_{t}\mathcal{E}$ as $s$ goes to infinity (since, by Theorem \ref{thm:ASCVNt}, $\frac{N_{s}}{W(s)}$ converge in $L^{2}$ and using Cauchy-Schwarz inequality), we can exchange limit and integrals to obtain,
\begin{align*}
\lim\limits_{s\to\infty}\mathbb{E}N_{t}\frac{N_{s}}{W(s)}=\underbrace{\mathbb{E}N_{t}\mathcal{E}}_{:=f(t)}=&\underbrace{\int_{[0,t]}\mathbb{E}\left[N_{t-u}\mathcal{E}\right]e^{-\alpha u}\ \mathbb{P}\left(V>u\right)\ b \ du}_{=:f\star G(t)}\\
&+\underbrace{\int_{[0,t]\times[0,\infty)}\alpha b\mathbb{E}\left[N_{t-u}\right] e^{-\alpha v}\mathbb{P}\left(V>u,V>v\right)\ du\ dv}_{=:\zeta_{1}(t)}\\
&+\underbrace{\int_{[0,\infty]}\alpha e^{-\alpha v}\mathbb{P}\left(V>v,V>t\right)\ dv}_{=:\zeta_{2}(t)},
\end{align*}
where we used that $\lim\limits_{t\to\infty}W(t)^{-1}\mathbb{E}N_{t}=\frac{\alpha}{b}$.

Now, we need to solve the last equation to obtain the last part of the lemma. To do that, we compute the Laplace transform of each part of the equation. Note that, since $W(t)=\mathcal{O}\left(e^{\alpha t}\right)$, it is easy to see that the Laplace transform of each term of \eqref{eq:renewNt} is well-defined as soon as $\lambda>\alpha$ (using Cauchy-Schwarz inequality for the first term).
Now, using \eqref{eq:laplacepv},
\begin{align}
T_{\mathcal{L}}e^{\alpha \cdot}G(\lambda)&=b\int_{\mathbb{R}_{+}}e^{-\lambda t}\mathbb{P}\left(V>t\right)dt=b\int_{\mathbb{R}_{+}}e^{-\lambda t}\int_{(t,\infty)}\mathbb{P}_{V}\left(dv\right)dt\nonumber\\
&=\frac{1}{\lambda}\int_{\mathbb{R}_{+}}\left(1-e^{-\lambda v}\right)\ b\mathbb{P}_{V}\left(dv\right)
=1-\frac{\psi(\lambda)}{\lambda}\label{eq:shortrangeeq}.
\end{align}
So,
\[
T_{\mathcal{L}}G(\lambda)=1-\frac{\psi(\lambda+\alpha)}{\lambda+\alpha}.
\]
Then,
\begin{multline*}
T_{\mathcal{L}}\zeta_{1}(\lambda)=\alpha T_{\mathcal{L}}\mathbb{E}N_{.}(\lambda)T_{\mathcal{L}}\left(b \int_{\mathbb{R}_{+}}e^{-\alpha v}\mathbb{P}\left(V>\cdot,V>v\right)dv\right)(\lambda)\\=\left(\frac{\lambda}{\psi(\lambda)}-1\right)\underbrace{T_{\mathcal{L}}\left(\alpha \int_{\mathbb{R}_{+}}e^{-\alpha v}\mathbb{P}\left(V>\cdot,V>v\right)dv\right)(\lambda)}_{=\mathcal{L}\zeta_{2}(\lambda)}.
\end{multline*}
and, using \eqref{eq:shortrangeeq}, we get
\begin{equation*}
T_{\mathcal{L}}\zeta_{2}(\lambda)=\alpha\int_{\mathbb{R}_{+}}e^{-\lambda t}\int_{\mathbb{R_{+}}}e^{-\alpha v}\mathbb{P}\left(V>t,V>v\right)dv\ dt
=\frac{1}{b}\left(\frac{\psi\left(\lambda+\alpha \right)}{\lambda}- \frac{\psi(\lambda)}{\lambda}\right).
\end{equation*}
Finally, we obtain,
\[
T_{\mathcal{L}}f(\lambda)=T_{\mathcal{L}}f(\lambda)\left(1-\frac{\psi(\lambda+\alpha)}{\lambda+\alpha} \right)+\left(\frac{\lambda}{\psi(\lambda)}-1\right)\frac{1}{b}\left(\frac{\psi\left(\lambda+\alpha \right)}{\lambda}- \frac{\psi(\lambda)}{\lambda}\right)+\frac{1}{b}\left(\frac{\psi\left(\lambda+\alpha \right)}{\lambda}- \frac{\psi(\lambda)}{\lambda}\right).
\]
Hence,
\[
T_{\mathcal{L}}f(\lambda)=\frac{\lambda}{b}\left(\frac{1}{\psi(\lambda)}-\frac{1}{\psi(\lambda+\alpha)} \right).
\]
Finally, using \eqref{eq:laplacescale} and
\[
bT_{\mathcal{L}}\left(W\star\mathbb{P}_{V}\right)(\lambda)=\frac{\left(\psi(\lambda)-b+\lambda \right)}{\psi(\lambda)},
\]
allows to inverse the Laplace transform of $f$ and get the result.
\end{proof}

Lemma \ref{lem:joinMom} allows us to compute the expected quadratic error.
\begin{lem}[Quadratic error in the convergence of $N_{t}$]
\label{lem:quadeR}
Let $\mathcal{E}$ the a.s.\ limit of $\psi'(\alpha)e^{-\alpha t}N_{t}$. Then,
\[
\lim\limits_{t\to\infty}e^{-\alpha t}\mathbb{E}\left(\psi'(\alpha)N_{t}-e^{\alpha t}\mathcal{E}\right)^{2}=\frac{\alpha}{b}\left(2-\psi'(\alpha) \right).
\]
\end{lem}
\begin{proof}
Let
\[
\mu:=\lim\limits_{t\to \infty}e^{\alpha t}F(t),
\]
where $F$ is defined in Proposition \ref{lem:WComp}.
We have, using Proposition \ref{lem:WComp} and \eqref{eq:intalpha},
\begin{align*}
\int_{[0,t]}W(t-u)\mathbb{P}_{V}(du)&=\frac{e^{\alpha t}}{\psi'(\alpha)}\left(1-\frac{\alpha}{b} \right)-\mu-\frac{e^{\alpha t}}{\psi'(\alpha)}\int_{(t,\infty)}e^{-\alpha u}\mathbb{P}_{V}(du)+\int_{[0,t]}\left(\mu-e^{\alpha (t-u)}F(t-u)\right)\mathbb{P}_{V}(du)\\
&=\frac{e^{\alpha t}}{\psi'(\alpha)}\left(1-\frac{\alpha}{b} \right)-\mu+o(1).
\end{align*}
Hence, the expression of $\mathbb{E}N_{t}\mathcal{E}$ given by Lemma \ref{lem:joinMom} can be rewritten, thanks to Lemmas \ref{lem:WComp}, as
\begin{equation}
\label{eq:joinMom}
\mathbb{E}N_{t}\mathcal{E}=\frac{2\alpha e^{\alpha t}}{b\psi'(\alpha)}-\frac{\alpha}{b}\left(\frac{1}{\psi'(\alpha)}+\mu \right)+o(1),
\end{equation}
Using \eqref{eq:loiNt} and \eqref{eq:probaNoCond} in conjunction with Proposition \ref{lem:WComp}, we also have
\begin{equation}
\label{eq:joinMomQuad}
e^{-\alpha t}\mathbb{E}N_{t}^{2}=2\frac{\alpha e^{\alpha t}}{b\psi'(\alpha)^{2}}-\frac{2\alpha\mu}{b\psi'(\alpha)}-\frac{\alpha}{b\psi'(\alpha)}+o(1).
\end{equation}
Hence, it finally follows from \eqref{eq:joinMom} and \eqref{eq:joinMomQuad} that
\begin{align*}
e^{-\alpha t}\mathbb{E}\left(\psi'(\alpha)N_{t}-e^{\alpha t}\mathcal{E}\right)^{2}&=\psi'(\alpha)^{2}e^{-\alpha t}\mathbb{E}N_{t}^{2}-2\psi'(\alpha)\mathbb{E}N_{t}\mathcal{E}+\frac{2\alpha e^{\alpha t}}{b}\\
&=-2\frac{\alpha\mu}{b}\psi'(\alpha)-\frac{\alpha\psi'(\alpha)}{b}+2\frac{\alpha}{b}\left(1+\psi'(\alpha)\mu \right)+o(1)\\
&=\frac{\alpha}{b}\left(2-\psi'(\alpha) \right)+o(1).
\end{align*}

\end{proof}

It is worth noting that, using \eqref{eq:probaNoCond} and the method above, we have the following result.
\begin{cor}
\label{cor:withCond}
We have
\begin{equation}
\label{eq:survie}
\frac{1}{\mathbb{P}\left(N_{t}>0\right)}=\frac{b}{\alpha}-\frac{b\mu\psi'(\alpha)}{\alpha}e^{-\alpha t}+o(e^{-\alpha t}),
\end{equation}
which leads to
 \begin{equation}
 \label{eq:AsNtE}
  \mathbb{E}_{t}N_{t}\mathcal{E}=\frac{2e^{\alpha t}}{\psi'(\alpha)}-\frac{1}{\psi'(\alpha)}-3\mu+o(1).
 \end{equation}
 
 \end{cor}
Our last estimate is the boundedness of the third moments.
\begin{lem}[Boundedness of the third moment]
\label{lem:NtBound}
The third moment of the error is asymptotically bounded, that is
\[
\mathbb{E}\left[\left|e^{-\frac{\alpha}{2} t}\left(\psi'(\alpha)N_{t}-e^{\alpha t}\mathcal{E}\right) \right|^{3} \right]=\mathcal{O}\left(1 \right).
\]
\end{lem}
\begin{proof}
We define for all $t\geq0$, $N^{\infty}_{t}$ as the number of individuals alive at time $t$ which have an infinite descent. According to Proposition 6.1 of \cite{CH}, $N^{\infty}$ is a Yule process under $\mathbb{P}_{\infty}$.

We have
\begin{align*}
\mathbb{E}\left[\left|\frac{\psi'(\alpha)N_{t}-e^{\alpha t}\mathcal{E}}{e^{\frac{\alpha}{2} t}} \right|^{3} \right]\leq 8\mathbb{E}\left[\left|\frac{\psi'(\alpha)N_{t}-N^{\infty}_{t}}{e^{\frac{\alpha}{2} t}} \right|^{3} \right]+8\mathbb{E}\left[\left|\frac{N^{\infty}_{t}-e^{\alpha t}\mathcal{E}}{e^{\frac{\alpha}{2} t}} \right|^{3} \right].
\end{align*}
Now, we know according to the proof of Theorem 6.2 of \cite{CH} (and this is easy to prove using the decomposition of Figure \ref{fig:subtrees}) that $N^{\infty}$ can be decomposed as
\[
N^{\infty}_{t}=\sum_{i=1}^{N_{t}}B_{i}^{(t)},
\]
where $\left(B_{i}^{(t)}\right)_{i\geq 1}$ is a family of independent Bernoulli random variables, which is i.i.d.\ for $i\geq 2$, under $\mathbb{P}_{t}$.
Hence,
\begin{align*}
\mathbb{E}_{t}\left[\left|\frac{\psi'(\alpha)N_{t}-N^{\infty}_{t}}{e^{\frac{\alpha}{2} t}} \right|^{3} \right]&\leq e^{-\frac{3}{2}\alpha t}\mathbb{E}_{t}\left[\left(\sum_{i=1}^{N_{t}}\left(\psi'(\alpha)-B_{i}^{(t)}  \right)\right)^{4}\right]^{\frac{3}{4}}.\\
\end{align*}
Since, it is known from the proof of Theorem 6.2 of \cite{CH} that
\[
\mathbb{E}B_{2}^{(t)}=\psi'(\alpha)+\mathcal{O}\left(e^{-\alpha t}\right),
\]
it is straightforward that 
\[
\mathbb{E}_{t}\left[\left|\frac{\psi'(\alpha)N_{t}-N^{\infty}_{t}}{e^{\frac{\alpha}{2} t}} \right|^{3} \right]
\]
is bounded.

On the other hand, we know that a Yule process is a time-changed Poisson process (see for instance \cite{AN}, Theorem III.11.2), that is, if $P_{t}$ is a Poisson process independent of $\mathcal{E}$ under $\mathbb{P}_{\infty}$,
\[
\mathbb{E}\left[\left|\frac{N^{\infty}_{t}-e^{\alpha t}\mathcal{E}}{e^{\frac{\alpha}{2} t}} \right|^{3}\right]=\mathbb{E}_{\infty}\left[\left|\frac{P_{\mathcal{E}\left(e^{\alpha t}-1\right)}-e^{\alpha t}\mathcal{E}}{e^{\frac{\alpha}{2} t}} \right|^{3}\right]\mathbb{P}(\text{NonEx}).
\]
Now, using Hlder inequality, it remains to bound
\[
 \mathbb{E}_{\infty}\left[\left(\frac{P_{\mathcal{E}\left(e^{\alpha t}-1\right)}-e^{\alpha t}\mathcal{E}}{e^{\frac{\alpha}{2} t}} \right)^{4}\right]=e^{-2\alpha t}\int_{\mathbb{R}_{+}}\mathbb{E}_{\infty}\left[\left(P_{x\left(e^{\alpha t}-1\right)}-e^{\alpha t}x \right)^{4}\right]e^{-x}dx.
\]
Finally, for a Poissonian random variable $X$ with parameter $\nu$, straightforward computations give that
$
\mathbb{E}\left[\left(X-\nu\right)^{4}\right]=3\nu^{2}+\nu,
$
which allows us to end the proof.
\end{proof}
\subsubsection{Case with arbitrary initial distribution $\mathbb{P}_{V_{\emptyset}}$}
\label{ssec:Ntarbi}
In order to study the behavior of the sub-splitting trees involved in the decomposition described in Figure \ref{fig:subtrees}, we investigate the behaviour of a splitting tree where the ancestor lifelength is not distributed as $V$, but follows an arbitrary distribution. Let $\Xi$ be a random variable in $(0,\infty]$, giving to the life-length of the ancestor and by $N(\Xi)$ the associated population counting process.

Using the decomposition of $N(\Xi)$ over the lifespan of the ancestor, as described in Section \ref{sec:strategy}, we have
\begin{equation}
\label{eq:ancBran}
N_{t}(\Xi)=\int_{\mathbb{R}_{+}}N^{(\xi_{u})}_{t-u}\mathds{1}_{\Xi>u}\ \xi(du)+\mathds{1}_{\Xi>t},
\end{equation}
where $\left(N^{i}\right)_{i\geq 1}$ is a family of i.i.d.\ CMJ processes with the same law as $N$ independent of $\Xi$ and $\xi$, as described in section \ref{sec:strategy}.
Let, for all $i\geq1$, $\mathcal{E}_{i}$ be
\begin{equation}
\label{eq:asCV2}
\mathcal{E}_{i}:=\lim\limits_{t\to\infty}\psi'(\alpha)e^{-\alpha t}N^{i}_{t}, \quad a.s,
\end{equation}
and, let $\mathcal{E}\left(\Xi\right)$  be the random variable defined by
\begin{equation}
\label{eq:ancBranlim}
\mathcal{E}\left(\Xi\right):=\int_{[0,\infty]}\mathcal{E}_{(\xi_{u})}e^{-\alpha u} \mathds{1}_{\Xi>u}\ \xi(du).
\end{equation}
\begin{lem}[First moment]
The first moment is asymptotically bounded, that is
\label{lem:firstMom}
\[
\mathbb{E}\left(\psi'(\alpha)N_{t}(\Xi)-e^{\alpha t}\mathcal{E}(\Xi)\right)=\mathcal{O}(1),
\]
uniformly with respect to the random variable $\Xi$.
\end{lem}
\begin{proof}
Using Lemma \ref{lem:stocInt}, \eqref{eq:ancBran} and \eqref{eq:ancBranlim} with have
\begin{equation*}
\mathbb{E}\left(\psi'(\alpha)N_{t}(\Xi)-e^{\alpha t}\mathcal{E}(\Xi)\right)=\int_{[0,t]}\left(\psi'(\alpha)\mathbb{E}N_{t-u}-e^{\alpha(t-u)}\mathbb{E}\mathcal{E}\right)e^{-\alpha u}\mathbb{P}\left( \Xi>u\right)bdu,
\end{equation*}
which leads using \eqref{eq:NtNoCond} and \eqref{eq:Prnonex} to
\begin{equation}
\label{eq:ptiteeq}
\mathbb{E}\left(\psi'(\alpha)N_{t}(\Xi)-e^{\alpha t}\mathcal{E}(\Xi)\right)=\int_{[0,t]}\underbrace{\left(\psi'(\alpha)W(t-u)-\psi'(\alpha)W\star\mathbb{P}_{V}(t-u)-\frac{\alpha}{b}e^{\alpha(t-u)}\right)}_{=:I_{t-u}}e^{-\alpha u}\mathbb{P}\left( \Xi>u\right)bdu.
\end{equation}
We get using Proposition \ref{lem:WComp} and \eqref{eq:intalpha},
\begin{align*}
\label{eq:truccc}
I_{s}=&e^{\alpha s}-\psi'(\alpha)e^{\alpha s}F(s)
-e^{\alpha s}\left(1-\frac{\alpha}{b} \right)\\&+\psi'(\alpha)\int_{[0,s]}e^{\alpha(s-v)}F(s-v)\mathbb{P}_{V}(dv)+e^{\alpha s}\int_{(s,\infty)}e^{-\alpha v}\mathbb{P}_{V}(dv)
-\frac{\alpha}{b}e^{\alpha s}\\
=&e^{\alpha s}\int_{(s,\infty)}e^{-\alpha v}\mathbb{P}_{V}(dv)+o(1).
\end{align*}
Hence, $\left(I_{s}\right)_{s\geq 0}$ is bounded. The result, now, follows from \eqref{eq:ptiteeq}.
\end{proof}
\begin{lem}[$L^{2}$ convergence in the general case]
\label{lem:genConv}
 $\psi'(\alpha)e^{-\alpha t}N_{t}(\Xi)$ converge a.s.\ and in $L^{2}$ to $\mathcal{E}\left(\Xi\right)$, and
\[
\lim\limits_{t\to\infty}e^{-\alpha t}\mathbb{E}\left(\psi'(\alpha)N_{t}(\Xi)-e^{\alpha t}\mathcal{E}(\Xi)\right)^{2}=\frac{\alpha}{b}\left(2-\psi'(\alpha)\right)\int_{ \mathbb{R}_{+}}e^{-\alpha s}\mathbb{P}\left(\Xi>s \right)bds,
\]
where the convergence is uniform with respect to $\Xi$ in $(0,\infty]$.
In the particular case when $\Xi$ follows the distribution of $O_{2}^{(\beta t)}$ given by \eqref{eq:loiOversh}, we have, for $0<\beta<\frac{1}{2}$,
\[
\lim\limits_{t\to\infty}e^{\alpha t}\mathbb{E}_{\beta t}\left(e^{-\alpha t}\psi'(\alpha)N_{t}(O^{(\beta t)}_{2})-\mathcal{E}(O^{(\beta t)}_{2})\right)^{2}=\left(2-\psi'(\alpha)\right)\psi'(\alpha).
\]
\end{lem}
\begin{proof}
From \eqref{eq:ancBran} and \eqref{eq:ancBranlim}, we have
\begin{equation}
\label{eq:ancExp}
\left(e^{-\alpha t}\psi'(\alpha)N_{t}(\Xi)-\mathcal{E}(\Xi)\right)^{2}=\left[\int_{\mathbb{R}_{+}}\left(e^{-\alpha (t-u)}\psi'(\alpha)N^{(\xi_{u})}_{t-u}-\mathcal{E}_{(u)}\right)e^{-\alpha u}\mathds{1}_{\Xi>u}\ \xi(du)+e^{-\alpha t}\mathds{1}_{\Xi>t}\right]^{2}
\end{equation}
and, using Lemma \ref{lem:stocInt},
\begin{align*}
&\mathbb{E}\left(\psi'(\alpha)e^{-\alpha t}N_{t}(\Xi)-\mathcal{E}(\Xi)\right)^{2}\\
=&\mathbb{E}\left(\int_{\mathbb{R}_{+}}\left(\psi'(\alpha)e^{-\alpha (t-u)}N^{(\xi_{u})}_{t-u}-\mathcal{E}_{(u)}\right)e^{-\alpha u}\mathds{1}_{\Xi>u}\ \xi(du)\right)^{2}\\&+e^{-2\alpha t}\mathbb{P}\left(\Xi>t\right)+2e^{-\alpha t}\mathbb{E}\mathds{1}_{\Xi>t}\int_{\mathbb{R}_{+}}\left(\psi'(\alpha)e^{-\alpha (t-u)}N^{(\xi_{u})}_{t-u}-\mathcal{E}_{(u)}\right)e^{-\alpha u}\mathds{1}_{\Xi>u}\ \xi(du),\\
=&\int_{\mathbb{R}_{+}}\mathbb{E}\left[\left(\psi'(\alpha)e^{-\alpha (t-u)}N^{(\xi_{u})}_{t-u}-\mathcal{E}_{(u)}\right)^{2}\right]e^{-2\alpha u}\mathbb{P}\left(\Xi>u\right)\ bdu\\
&+\int_{\mathbb{R}_{+}}\mathbb{E}\left(\psi'(\alpha)e^{-\alpha (t-u)}N^{(\xi_{u})}_{t-u}-\mathcal{E}_{(u)}\right)\mathbb{E}\left(\psi'(\alpha)e^{-\alpha (t-v)}N^{(\xi_{v})}_{t-v}-\mathcal{E}_{(v)}\right)e^{-\alpha (u+v)}\mathbb{P}\left(\Xi>u,\Xi>v\right)\ bdu\ dv\\
&+e^{-2\alpha t}\mathbb{P}\left(\Xi>t\right)+2e^{-\alpha t}\int_{\mathbb{R}_{+}}\mathbb{E}\left(\psi'(\alpha)e^{-\alpha (t-u)}N^{(\xi_{u})}_{t-u}-\mathcal{E}_{(u)}\right)e^{-\alpha u}\mathbb{P}\left(\Xi>u,\Xi>t\right)\ bdu.
\end{align*}
Moreover, since,
\[
\psi'(\alpha)\mathbb{E}e^{-\alpha t}N_{t}-\mathcal{E}=\mathcal{O}\left(e^{-\alpha t}\right),
\]
this leads, using Lemma \ref{lem:firstMom}, to
\[
\lim\limits_{t\to\infty}e^{\alpha t}\mathbb{E}\left(e^{-\alpha t}\psi'(\alpha)N_{t}(\Xi)-\mathcal{E}(\Xi)\right)^{2}=\frac{\alpha}{b}\left(2-\psi'(\alpha)\right)\int_{ \mathbb{R}_{+}}e^{-\alpha u}\mathbb{P}\left(\Xi>u \right)bdu.
\]
Now, we have from \eqref{eq:loiOversh} and Lemma \ref{lem: asyComp},
\[
\lim\limits_{u\to\infty}\mathbb{P}_{u}\left(O_{2}>s \right)=\lim\limits_{u\to\infty}\int_{\mathbb{R}_{+}}\frac{W(u-y)}{W(u)-1}\mathbb{P}\left(V>s+y \right)bdy=\int_{\mathbb{R}_{+}}e^{-\alpha y}\mathbb{P}\left(V>s+y \right)bdy.
\]
It follows then from Lebesgue theorem that,
\[
\lim\limits_{t\to\infty}\int_{ \mathbb{R}_{+}}e^{-\alpha s}\mathbb{P}_{\beta t}\left(O_{2}>s \right)bds=\frac{b\psi'(\alpha)}{\alpha}.\qedhere
\]
\end{proof}

\begin{lem}[Boundedness in the general case.]
\label{lem:genBounded}
The error of order $3$ in asymptotically bounded, that is
\[
e^{-\frac{3}{2}\alpha t}\mathbb{E}\left|\psi'(\alpha)N_{t}(\Xi)-e^{\alpha t}\mathcal{E}(\Xi)\right|^{3}=\mathcal{O}\left(1\right),
\]
 uniformly w.r.t. $\Xi$.
\end{lem}
\begin{proof}
Rewriting  $N(\Xi)$ and $\mathcal{E}\left(\Xi \right)$ as in the proof of Lemma \ref{lem:genConv}, we see that,
\begin{align*}
e^{-\frac{3}{2}t}\ &\mathbb{E}\left|\psi'(\alpha)N_{t}(\Xi)-e^{\alpha t}\mathcal{E}(\Xi)\right|^{3}=e^{-\frac{3}{2}t}\ \mathbb{E} \left[\left|\int_{[0,t]}\left(\psi'(\alpha)N^{(\xi_{u})}_{t-u}-e^{\alpha (t-u)}\mathcal{E}_{(u)}\right)\mathds{1}_{\Xi>u}\ \xi(du)+\psi'(\alpha)\mathds{1}_{\Xi>t}\right|^{3}\right]\\
&\leq 8 \mathbb{E}\left|\int_{[0,t]}e^{-\frac{3}{2}(t-u)}\left(\psi'(\alpha)N^{(\xi_{u})}_{t-u}-e^{\alpha (t-u)}\mathcal{E}_{(u)}\right)\ e^{-\frac{1}{2}u}\mathds{1}_{\Xi>u} \xi(du)\right|^{3}+8\psi'(\alpha)e^{-\frac{1}{2}t}\mathbb{P}\left(\Xi>t\right)^{3}\\
\end{align*}
We denote by $I$ the first term of the r.h.s. of the last inequality, leading to
\begin{align*}
I&\leq 8 \mathbb{E}\int_{[0,t]^{3}}\prod_{i=1}^{3}\left|e^{-\frac{1}{2}(t-s_{i})}\left(\psi'(\alpha)N^{(\xi_{s_{i}})}_{t-s_{i}}-e^{\alpha (t-s_{i})}\mathcal{E}_{(s_{i})}\right)\right|\ e^{-\frac{1}{2}s_{i}}\mathds{1}_{\Xi>s_{i}} \xi(ds_{1})\xi(ds_{2})\xi(ds_{3})\\
&\leq 8 \mathbb{E}\int_{[0,t]^{3}}\sum_{j=1}^{3}\left|e^{-\frac{1}{2}(t-s_{j})}\left(\psi'(\alpha)N^{(\xi_{s_{j})}}_{t-s_{j}}-e^{\alpha (t-s_{j})}\mathcal{E}_{(s_{j})}\right)\right|^{3}\ \prod_{i=1}^{3}e^{-\frac{1}{2}s_{i}}\mathds{1}_{\Xi>s_{i}} \xi(ds_{1})\xi(ds_{2})\xi(ds_{3})\\
&\leq 24 \mathbb{E}\int_{[0,t]}\left|e^{-\frac{1}{2}(t-u)}\left(\psi'(\alpha)N^{(\xi_{u})}_{t-u}-e^{\alpha (t-u)}\mathcal{E}_{(u)}\right)\right|^{3}\ e^{-\frac{1}{2}u}\mathds{1}_{\Xi>u} \xi(du)\left(\int_{[0,t]}e^{-\frac{1}{2}u}\xi(du) \right)^{2}\\
&\leq 24 \mathbb{E}\int_{[0,t]}\left|e^{-\frac{1}{2}(t-u)}\left(\psi'(\alpha)N^{(\xi_{u})}_{t-u}-e^{\alpha (t-u)}\mathcal{E}_{(u)}\right)\right|^{3}\ e^{-\frac{1}{2}u}\mathds{1}_{\Xi>u}\ \mu(du),\\
\end{align*}
with
\[
\mu(du)=\left(\int_{[0,t]}e^{-\frac{1}{2}s}\xi(ds) \right)^{2}\xi(du).
\]
Now, since $\mu$ is independent from the family $\left(N^{(i)}\right)$ and $\left(\mathcal{E}_{(i)}\right)$, an easy adaptation of the proof of Lemma \ref{lem:stocInt}, leads to
\begin{align*}
e^{-\frac{3}{2}t}\ \mathbb{E}\left|\psi'(\alpha)N_{t}(\Xi)-e^{\alpha t}\mathcal{E}(\Xi)\right|^{3}\leq&24 \mathbb{E}\int_{[0,t]}\mathbb{E}\left[\left|e^{-\frac{1}{2}(t-u)}\left(\psi'(\alpha)N_{t-u}-e^{\alpha (t-u)}\mathcal{E}\right)\right|^{3}\right]\ e^{-\frac{1}{2}u}\mathds{1}_{\Xi>u}\ \mu(du)\\&+8\psi'(\alpha)e^{-\frac{1}{2}t}\mathbb{P}\left(\Xi>t\right)
\end{align*}
Using Lemma \ref{lem:NtBound} to bound
\[
\mathbb{E}\left|e^{-\frac{3}{2}(t-u)}\left(N_{t-u}-e^{\alpha (t-u)}\mathcal{E}\right)\right|^{3},
\]
in the previous expression, finally leads to
\[
e^{-\frac{3}{2}t}\ \mathbb{E}\left|\psi'(\alpha)N_{t}(\Xi)-e^{\alpha t}\mathcal{E}(\Xi)\right|^{3}\leq  \mathcal{C}\ \left(\mathbb{E}\left(\int_{\mathbb{R}_{+}}e^{-\frac{1}{2}u}\xi(du) \right)^{3} \ +1\right),
\]
for some real positive constant $\mathcal{C}$.
\end{proof}

\subsection{Proof of Theorem \ref{thm:tclN}}\label{sec:endProof}
We fix a positive real number $u$.
From this point, we recall the decomposition of the splitting tree as described in Section \ref{sec:strategy} (see also Figure \ref{fig:subtrees}).
We also recall that, for all $i$ in $\left\{ 1,\dots,N_{u}\right\}$, the process $\left(N^{i}_{s}\left(O_{i}\right),\ s\in\mathbb{R}_{+}\right)$ is the population counting process of the (sub-)splitting tree $\mathbb{T}\left(O_{i}\right)$.

As explained in Section \ref{sec:strategy}, it follows from the construction of the splitting tree, that, for all $i$ in $\left\{ 1,\dots,N_{u}\right\}$, there exists an i.i.d.\ family of processes $\left(N^{i,j}\right)_{j\geq 1}$ independent from $N_{u}$ with the same law as $\left(N_{t},\ t\in\mathbb{R}_{+}\right)$, and an i.i.d.\ family $\left(\xi^{(i)}\right)_{1\leq i\leq N_{u}}$ of random measure independent from $N_{u}$ and from $\left(N^{i,j}\right)_{j\geq 1}$ the family with same law as $\xi$, such that
\begin{equation}
\label{eq:definedec}
N_{t}^{i}\left(O_{i}\right)=\int_{[0,t]}N^{i,j}_{t-u}\mathds{1}_{O_{i}> u}\ \xi^{(i)}(du)+\mathds{1}_{O_{i}> t},\quad \forall t\in\mathbb{R}_{+},\quad \forall i\in\left\{ 1,\dots,N_{u}\right\}.
\end{equation}
As in \eqref{eq:ancBranlim}, we define, for all $i$ in $\left\{ 1,\dots,N_{u}\right\}$,
\begin{equation}
\label{eq:defineexp}
\mathcal{E}\left(O_{i}\right):=\int_{[0,t]}\mathcal{E}_{i,\xi^{(i)}_{u}}e^{-\alpha u}\mathds{1}_{O_{i}> u}\ \xi^{(i)}(du),
\end{equation}
where $\mathcal{E}_{i,j}:=\lim\limits_{t\to\infty}\psi'(\alpha)e^{-\alpha t}N^{i,j}_{t}$.

Hence, it follows from Lemma \ref{lem:genConv}, that $e^{-\alpha t}N^{i}_{t}\left(O_{i}\right)$ converges to $\mathcal{E}\left(O_{i}\right)$ in $L^{2}$.

Note also that, from Lemma \ref{lem:residual}, the family $\left(N^{i}_{t}\left(O_{i}\right),\ t\in\mathbb{R}_{+}\right)_{2\leq i\leq N_{u}}$ is i.i.d.\ and independent from $N_{u}$ under $\mathbb{P}_{u}$, as well as the family $\left(\mathcal{E}\left(O_{i}\right)\right)_{2\leq i\leq N_{u}}$ (in the sense of Remark \ref{rem:indep}). Note that the law under $\mathbb{P}_{u}$ of the processes of the family $\left(N^{i}_{t}\left(O_{i}\right),\ t\in\mathbb{R}_{+}\right)_{2\leq i\leq N_{u}}$ is the law of standard population counting processes where the lifespan of the root is distributed as $O_{2}$ under $\mathbb{P}_{u}$ (except for the first one).

\begin{lem}[Decomposition of $\mathcal{E}$]
\label{lem:dec}
We have the following decomposition of $\mathcal{E}$,
\[
\mathcal{E}=e^{-\alpha u}\sum_{i=1}^{N_{u}}\mathcal{E}_{i}\left(O_{i}\right), \quad a.s.
\]
Moreover, under $\mathbb{P}_{u}$, the random variables $\left(\mathcal{E}_{i}\left(O_{i}\right)\right)_{i\geq 1}$ (defined by \eqref{eq:defineexp}) are independent, independent of $N_{u}$, and identically distributed for $i\geq 2$.
\end{lem}
%
%
%
%
\begin{proof}

{\bf Step 1: Decomposition of $\mathcal{E}$.}

For all $t$ in $\mathbb{R}_{+}$, we denote by $N^{\infty}_{t}$ the number of individuals alive at time $t$ which have an infinite descent. For all $i$, we define, for all $t\geq 0$, $N^{\infty}_{t}\left(O_{i}\right)$ from $\mathbb{T}\left(O_{i}\right)$ as $N^{\infty}_{t }$ was defined from the whole tree.
Now, it is easily seen that
\[
N^{\infty}_{t}=\sum_{i=1}^{N_{u}}N^{\infty}_{t-u}\left(O_{i}\right).
\]
Hence, if $e^{-\alpha t}N^{\infty}_{t}\left(O_{i}\right)$ converges a.s.\ to $\mathcal{E}\left(O_{i}\right)$, then
\[
\lim\limits_{t\to\infty }e^{-\alpha t}N^{\infty}_{t}=\lim\limits_{t\to\infty}e^{-\alpha u}\sum_{i=1}^{N_{u}}e^{-\alpha(t-u)}N^{\infty}_{t-u}\left(O_{i}\right)=e^{-\alpha u}\sum_{i=1}^{N_{u}}\mathcal{E}\left(O_{i}\right).
\]
So, it just remains to prove the a.s.\ convergence to get the desired result.

{\bf Step 2: a.s.\ convergence of $N^{\infty}\left(O_{i}\right)$ to $\mathcal{E}\left(O_{i}\right)$.}

For this step, we fix $i\in \left\{1,\dots, N_{u}\right\}$.

In the same spirit as \eqref{eq:definedec} (see also Section \ref{sec:strategy}), it follows from the construction of the splitting tree $\mathbb{T}\left(O_{i}\right)$, that there exists, an i.i.d.\ (and independent of $N_{u}$) sequence of processes $\left(N^{j,\infty}_{s},\ s\in\mathbb{R}_{+}\right)_{j\geq 1}$ with the same law as $\left(N^{\infty}_{t},\ t\in\mathbb{R}_{+}\right)$ (under $\mathbb{P}$), such that
\[
N^{\infty}_{t}\left(O_{i}\right)=\int_{[0,t]}N^{\xi^{(i)}_{u},\infty}_{t-u}\ \mathds{1}_{O_{i}>u}\ \xi^{(i)}(du)+\mathds{1}_{O_{i}=\infty},\ \forall t\geq 0.
\]
Now, it follows from Theorem 6.2 of \cite{CH}, that for all $j$,
\[
\lim\limits_{t\to\infty}e^{-\alpha t}N^{j,\infty}_{t}=\mathcal{E}_{i,j},\ a.s.,
\]
where $\mathcal{E}_{i,j}$ was defined in the beginning of this section.
Let
\[
\mathcal{C}_{j}:=\sup_{t\in\mathbb{R}_{+}}e^{-\alpha t }N^{j,\infty}_{t}, \quad \forall j\geq 1,
\]
and
\[
\mathcal{C}:=\sup_{t\in\mathbb{R}_{+}}e^{-\alpha t }N^{\infty}_{t}.
\]
Then, the family $\left(\mathcal{C}_{j}\right)_{j\geq 1}$ is i.i.d., since the processes $\left(N^{j,\infty}\right)_{j\geq 1}$ are i.i.d, with the same law as $\mathcal{C}$.
Hence,
\begin{equation}
\label{eq:dixmil}
\int_{[0,t]}e^{-\alpha(t-u)}N^{\xi^{(i)}_{u},\infty}_{t-u}\ e^{-\alpha u}\mathds{1}_{O_{i}>u}\ \xi^{(i)}(du)\leq\int_{[0,t]}\mathcal{C}_{\xi^{(i)}_{u}} e^{-\alpha u}\mathds{1}_{O_{i}>u}\ \xi^{(i)}(du).
\end{equation}
It is easily seen that $\mathbb{E}\left[\mathcal{C}\right]=\mathbb{P}\left(\text{NonEx}\right)\mathbb{E}_{\infty}\left[C\right]$.
Now, since, from Proposition 6.1 of \cite{CH}, $N^{\infty}_{t}$ is a Yule process under $\mathbb{P}_{\infty}$ (and hence $e^{-\alpha t}N^{\infty}_{t}$ is a martingale),  Doobs's inequalities entails that the random variable $\mathcal{C}$ is integrable. Hence, the right hand side of the \eqref{eq:dixmil} is a.s.\ finite, and we can apply Lesbegue Theorem to get
\[
\lim\limits_{t\to\infty}e^{-\alpha t}N^{\infty}_{t}\left(O_{i}\right)=\int_{[0,t]}\mathcal{E}_{i,\xi^{(i)}_{u}}\ e^{-\alpha u}\mathds{1}_{O_{i}>u}\ \Gamma(du)=\mathcal{E}\left(O_{i}\right),\quad a.s.,
\]
where the right hand side of the last equality is just the definition of $\mathcal{E}\left(O_{i}\right)$.

\end{proof}

We have now all the tools needed to prove the central limit theorem for $N_{t}$.
\begin{proof}[Proof of Theorem \ref{thm:tclN}]
Let $u<t$, two positive real numbers.
From Lemma \ref{lem:dec} and section \ref{sec:strategy}, we have
\[
N_{t}=\sum_{i=1}^{N_{u}}N^{(i)}_{t-u}\left(O_{i}\right)
\]
and
\[
e^{\alpha t}\mathcal{E}=\sum_{i=1}^{N_{u}}e^{\alpha (t-u)}\mathcal{E}_{i}\left(O_{i}\right).
\]
Then,
\begin{equation}
\label{eq:decompTclN}
\frac{\psi'(\alpha)N_{t}-e^{\alpha t}\mathcal{E}}{e^{\frac{\alpha}{2} t}}=\sum_{i=1}^{N_{u}}\frac{\psi'(\alpha)N^{(i)}_{t-u}\left(O_{i}\right)-e^{\alpha (t-u)}\mathcal{E}_{i}\left(O_{i}\right)}{e^{\frac{\alpha}{2} (t-u)}e^{\frac{\alpha}{2} u}}.
\end{equation}
Using Lemma \ref{lem:residual}, we know that, under $\mathbb{P}_{u}$, $\left(N^{i}_{t-u}(O_{i}), t>u\right)_{ 1\leq i\leq N_{u}}$ are independent processes, i.i.d.\ for $i\geq2$ and independent of $N_{u}$. Let us denote by $\varphi$ and $\tilde{\varphi}$ the characteristic functions
\[
\varphi(\lambda):=\mathbb{E}\left[\exp\left({i\lambda \left(\frac{\psi'(\alpha)N^{2}_{t-u}\left(O_{2}\right)-e^{\alpha (t-u)}\mathcal{E}_{2}\left(O_{2}\right)}{e^{\frac{\alpha}{2} (t-u)}} \right)}\right) \right],\quad \lambda\in\mathbb{R}
\]
and
\[
\tilde{\varphi}(\lambda):=\mathbb{E}\left[\exp\left({i\lambda \left(\frac{\psi'(\alpha)N^{1}_{t-u}\left(O_{1}\right)-e^{\alpha (t-u)}\mathcal{E}_{1}\left(O_{1}\right)}{e^{\frac{\alpha}{2} (t-u)}} \right)}\right) \right], \quad \lambda\in\mathbb{R}.
\]
It follows from \eqref{eq:decompTclN} and Lemma \ref{lem:residual} that,
\begin{align*}
\mathbb{E}_{u}\left[\exp\left(i\lambda \frac{\psi'(\alpha)N_{t}-e^{\alpha t}\mathcal{E}}{e^{\frac{\alpha}{2} t}}\right) \right]
&=\frac{\tilde{\varphi}\left(\frac{\lambda}{e^{\frac{\alpha}{2} u}}\right)}{\varphi\left(\frac{\lambda}{e^{\frac{\alpha}{2} u}} \right)}\mathbb{E}_{u}\left[\varphi\left(\frac{\lambda}{e^{\frac{\alpha}{2} u}} \right)^{N_{u}} \right]
\end{align*}
Since $N_{u}$ is geometric with parameter $W(u)^{-1}$ under $\mathbb{P}_{u}$,
\[
\mathbb{E}_{u}\left[\exp\left(i\lambda \frac{\psi'(\alpha)N_{t}-e^{\alpha t}\mathcal{E}}{e^{\frac{\alpha}{2} t}}\right) \right]=\frac{\tilde{\varphi}\left(\frac{\lambda}{e^{\frac{\alpha}{2} u}}\right)}{\varphi\left(\frac{\lambda}{e^{\frac{\alpha}{2} u}} \right)}\frac{W(u)^{-1}\varphi\left(\frac{\lambda}{e^{\frac{\alpha}{2} u}} \right)}{1-\left(1-W(u)^{-1}\right)\varphi\left(\frac{\lambda}{e^{\frac{\alpha}{2} u}} \right)}
\]
Using Taylor formula for $\varphi$, we obtain,
\begin{align*}
\mathbb{E}_{u}\left[\exp\left(i\lambda\frac{\psi'(\alpha)N_{t}-e^{\alpha t}\mathcal{E}}{e^{\frac{\alpha}{2} t}}\right) \right]&=\tilde{\varphi}\left(\frac{\lambda}{e^{\frac{\alpha}{2} u}}\right)\frac{1}{D(\lambda,t,u)}
\end{align*}
where,
\begin{align*}
&D(\lambda,t,u)= W(u)\\&-\left(W(u)-1\right)\Bigg(1+i\lambda\mathbb{E}\left[\frac{\psi'(\alpha)N^{i}_{t-u}\left(O_{2}\right)-e^{\alpha (t-u)}\mathcal{E}_{2}\left(O_{2}\right)}{e^{\frac{\alpha}{2} (t-u)}e^{\frac{\alpha}{2} u}}\right]\\&\qquad\qquad\qquad\qquad\qquad\qquad\qquad\qquad\qquad-\frac{\lambda^{2}}{2}\mathbb{E}\left[\left(\frac{\psi'(\alpha)N^{i}_{t-u}\left(O_{2}\right)-e^{\alpha (t-u)}\mathcal{E}_{2}\left(O_{2}\right)}{e^{\frac{\alpha}{2} (t-u)}e^{\frac{\alpha}{2} u}}\right)^{2}\right]+R(\lambda,t,u) \Bigg)\\
&=1-i\lambda\frac{W(u)-1}{e^{\frac{\alpha}{2} u}}\mathbb{E}\left[\frac{\psi'(\alpha)N^{i}_{t-u}\left(O_{2}\right)-e^{\alpha (t-u)}\mathcal{E}_{2}\left(O_{2}\right)}{e^{\frac{\alpha}{2} (t-u)}}\right]\\&\qquad\qquad\qquad\qquad\qquad\qquad\qquad\qquad\qquad+\frac{\lambda^{2}}{2}\frac{W(u)-1}{e^{\alpha u}}\mathbb{E}\left[\left(\psi'(\alpha)\frac{N^{i}_{t-u}\left(O_{2}\right)-e^{\alpha (t-u)}\mathcal{E}_{2}\left(O_{2}\right)}{e^{\frac{\alpha}{2} (t-u)}}\right)^{2}\right]\\
&\qquad\qquad\qquad\qquad\qquad\qquad\qquad\qquad\qquad\qquad\qquad\qquad\qquad\qquad\qquad\qquad\qquad\qquad-(W(u)-1)R(\lambda,t,u),
\end{align*}

with, for all $\epsilon>0$ and all $\lambda$ in $(-\epsilon,\epsilon)$,
\begin{equation}
\label{eq:restEst}
\left|R(\lambda,t,u)\right|\leq\sup_{\lambda\in(-\epsilon,\epsilon)}\left|\frac{\partial^{3}}{\partial\lambda^{
3}}\varphi(\lambda) \right|\leq \mathbb{E}\left[\left|\left(\frac{\psi'(\alpha)N^{i}_{t-u}\left(O_{2}\right)-e^{\alpha (t-u)}\mathcal{E}_{2}\left(O_{2}\right)}{e^{\frac{\alpha}{2} (t-u)}}\right) \right|^{3} \right]\frac{\epsilon^{3} e^{-\frac{3}{2}\alpha u}}{6}\leq C\epsilon^{3}e^{-\frac{3}{2}u},
\end{equation}
for some real positive constant $C$ obtained using Lemma \ref{lem:genBounded}.

From this point, we set $u=\beta t$ with $0<\beta<\frac{1}{2}$. It follows then from the Lemmas \ref{lem:genConv} and \ref{lem:residual}, that
\begin{equation}
\label{eq:lim1}
\lim\limits_{t\to\infty}\mathbb{E}_{\beta t}\left[\left(\frac{\psi'(\alpha)N^{i}_{t-\beta t}\left(O_{2}\right)-e^{\alpha (t-\beta t)}\mathcal{E}_{2}\left(O_{2}\right)}{e^{\frac{\alpha}{2} (t-\beta t)}}\right)^{2}\right]=\psi'(\alpha)\left(2-\psi'(\alpha)\right).
\end{equation}
Moreover, we have from Lemma \ref{lem:firstMom}, and since $\beta<\frac{1}{2}$,
\begin{equation}
\label{eq:lim2}
\lim\limits_{t\to\infty}W(\beta t)e^{-\frac{\alpha}{2}t}\mathbb{E}\left[\psi'(\alpha)N^{i}_{t}\left(O_{2}\right)-e^{\alpha t}\mathcal{E}_{2}\left(O_{2}\right)\right]=0.
\end{equation}
Finally, the relations \eqref{eq:restEst}, \eqref{eq:lim1} and \eqref{eq:lim2} lead to
\[
\lim\limits_{t\to\infty}\mathbb{E}_{\beta t}\left[\exp\left(i\lambda \frac{N_{t}-e^{\alpha t}\mathcal{E}}{e^{\frac{\alpha}{2} t}}\right) \right]=\frac{1}{1+\frac{\lambda^{2}}{2}\left(2-\psi'(\alpha)\right)}.
\]
To conclude, note that,
\begin{align*}
\left|\mathbb{E}_{\beta t}\left[\exp\left(i\lambda \frac{N_{t}-e^{\alpha t}\mathcal{E}}{e^{\frac{\alpha}{2} t}}\right)\right]-\mathbb{E}_{\infty}\left[\exp\left(i\lambda \frac{N_{t}-e^{\alpha t}\mathcal{E}}{e^{\frac{\alpha}{2} t}}\right) \right]\right|&=\left|\mathbb{E}\left[e^{i\lambda \frac{\psi'(\alpha)N_{t}-e^{\alpha t}\mathcal{E}}{e^{\frac{\alpha}{2} t}}} \left(\frac{\mathds{1}_{N_{\beta t}>0}}{\mathbb{P}\left(N_{\beta t}>0\right)}-\frac{\mathds{1}_{\text{NonEx}}}{\mathbb{P}\left(\text{NonEx} \right)} \right) \right]\right|\\
&\leq \mathbb{E}\left[ \left|\frac{\mathds{1}_{N_{\beta t}>0}}{\mathbb{P}\left(N_{\beta t}>0\right)}-\frac{\mathds{1}_{\text{NonEx}}}{\mathbb{P}\left(\text{NonEx} \right)} \right| \right]
\end{align*}
goes to $0$ as $t$ goes to infinity.
This ends the proof of Theorem \ref{thm:tclN}.
\end{proof}

\bibliographystyle{plain}
\bibliography{biblio}
\end{document}